\documentclass[11pt,a4paper,leqno,noamsfonts]{amsart}
   \makeatletter 
   \renewcommand\@biblabel[1]{#1.}
   \makeatother
\usepackage[english]{babel}
\usepackage[dvipsnames]{xcolor}
\usepackage{graphicx,pifont,soul}
\usepackage[utopia]{mathdesign}
\usepackage[a4paper,top=1.8cm,bottom=1.8cm,
    left=1.8cm,right=1.8cm,marginparwidth=60pt]{geometry}
\usepackage[utf8]{inputenc}
\usepackage{braket,caption,comment,mathtools,stmaryrd}
\usepackage[usestackEOL]{stackengine}
\usepackage{multirow,booktabs,microtype,relsize}
\usepackage{tikz,tikz-cd}
    \numberwithin{equation}{section}
\usepackage[colorlinks,bookmarks]{hyperref} %
  \hypersetup{colorlinks,%
              citecolor=britishracinggreen,%
              filecolor=black,%
              linkcolor=cobalt,%
              urlcolor=cornellred}
\usepackage[capitalise]{cleveref}
\DeclareSymbolFont{usualmathcal}{OMS}{cmsy}{m}{n}
\DeclareSymbolFontAlphabet{\mathcal}{usualmathcal}
\DeclareMathAlphabet\BCal{OMS}{cmsy}{b}{n}

\definecolor{cornellred}{rgb}{0.7, 0.11, 0.11}
\definecolor{britishracinggreen}{rgb}{0.0, 0.26, 0.15}
\definecolor{cobalt}{rgb}{0.0, 0.28, 0.67}
\usepackage{comment,braket}
\usepackage[shortlabels]{enumitem}

\theoremstyle{definition}
\newtheorem*{lemma*}{Lemma}
\newtheorem*{theorem*}{Theorem}
\newtheorem*{example*}{Example}
\newtheorem*{fact*}{Fact}
\newtheorem*{notation*}{Notation}
\newtheorem*{definition*}{Definition}
\newtheorem*{prop*}{Proposition}
\newtheorem*{remark*}{Remark}
\newtheorem*{corollary*}{Corollary}

\newtheorem*{conventions*}{Conventions}

\newtheorem{definition}{Definition}[section]

\newtheorem{example}[definition]{Example}

\newtheorem{remark}[definition]{Remark}
\newtheorem{conjecture}[definition]{Conjecture}

\newtheoremstyle{thm} 
        {3mm}
        {3mm}
        {\slshape}
        {0mm}
        {\bfseries}
        {.}
        {1mm}
        {}
\theoremstyle{thm}
\newtheorem{theorem}[definition]{Theorem}
\newtheorem{corollary}[definition]{Corollary}
\newtheorem{lemma}[definition]{Lemma}
\newtheorem{proposition}[definition]{Proposition}
\newtheorem{thm}{Theorem}

\newcommand{\simto}{\,\widetilde{\to}\,}
\newcommand{\boldit}[1]{\boldsymbol{#1}} 

\newcommand{\Z}{\mathbb{Z}}

\newcommand{\mQ}{\mathcal{Q}}

\newcommand{\mF}{\mathcal{F}}

\newcommand{\PP}{\mathbb{P}}

\newcommand{\mE}{\mathcal{E}}
\newcommand{\mU}{\mathcal{U}}

\newcommand{\of}{\mathcal{O}}

\DeclareMathOperator{\Spec}{Spec}

\DeclareMathOperator{\Sym}{Sym}

\DeclareMathOperator{\Gr}{Gr}
\DeclareMathOperator{\Fl}{Fl}

\DeclareMathOperator{\codim}{codim}

\DeclareMathOperator{\Bl}{Bl}

\DeclareMathOperator{\rank}{rank}
\DeclareMathOperator{\Hilb}{Hilb}
\DeclareMathOperator{\Var}{Var}
\DeclareMathOperator{\Exp}{Exp}

\newcommand*{\defeq}{\mathrel{\vcenter{\baselineskip0.5ex \lineskiplimit0pt
\hbox{\scriptsize.}\hbox{\scriptsize.}}}%
=}
\numberwithin{equation}{section}
    
\author{Enrico Fatighenti}
\address{\newline
Sapienza Universit\`a di Roma\hfill\newline
Dipartimento di Matematica ``Guido Castelnuovo''\hfill\newline
Piazzale Aldo Moro 5, 00185 Roma, Italy}
\email[E.~Fatighenti]{fatighenti@mat.uniroma1.it}

\author{Francesco Meazzini}
\address{\newline
Alma Mater studiorum Universit\`a di Bologna\hfill\newline
Dipartimento di Matematica\hfill\newline
Piazza di porta San Donato 5, 40126 Bologna, Italy}
\email[F.~Meazzini]{francesco.meazzini2@unibo.it}

\author{Giovanni Mongardi}
\address{\newline
Alma Mater studiorum Universit\`a di Bologna\hfill\newline
Dipartimento di Matematica\hfill\newline
Piazza di porta San Donato 5, 40126 Bologna, Italy}
\email[G.~Mongardi]{giovanni.mongardi2@unibo.it}

\author{Andrea T. Ricolfi}
\address{\newline
Alma Mater studiorum Universit\`a di Bologna\hfill\newline
Dipartimento di Matematica\hfill\newline
Piazza di porta San Donato 5, 40126 Bologna, Italy}
\email[A.~T.~Ricolfi]{andreatobia.ricolfi@unibo.it}

\title[Hilbert squares of degeneracy loci]{Hilbert squares of degeneracy loci}

\begin{document}

\begin{abstract}
Let $S$ be the first degeneracy locus of a morphism of vector bundles corresponding to a general matrix of linear forms in $\PP^s$.
We prove that, under certain positivity conditions, its Hilbert square $\Hilb^2(S)$ is isomorphic to the zero locus of a global section of an irreducible homogeneous vector bundle on a product of Grassmannians. 
Our construction involves a naturally associated Fano variety, and an explicit description of the isomorphism.
\end{abstract}

\maketitle


\thispagestyle{empty}

\section{Introduction}
The Hilbert scheme of 2 points $\Hilb^2(S)$ on a smooth variety $S$, called the \emph{Hilbert square of} $S$, is an interesting smooth variety, whose geometry is incredibly rich, and yet not fully understood. An intriguing problem consists in finding a projective embedding of $\Hilb^2(S)$, for example by either writing down equations or realising it as the zero locus of a section of some vector bundle. An archetypical example is when $S_g$ is a K3 surface of genus $g$, in which case $\Hilb^2(S_g)$ is a hyperk\"ahler fourfold, and a projective embedding is known in a bunch of cases, including $g=3,5,7,8,12$ --- the last one only up to deformations, see \cite{beauville1985variete, benedetti2018sous, debarre2010hyper,fatighenti2021fano, iliev2019hyperkaehler}. A few other cases are known, including the recent case of $\PP^2$, see \cite{hsm}.

In this paper, we focus on the special case where $S \subset \PP^s$ arises as the first degeneracy locus of a general morphism of vector bundles 
\[
\varphi \colon \mathcal O_{\mathbb P^s}^{\oplus n+m} \to \mathcal O_{\mathbb P^s}(1)^{\oplus n}.
\]
The case $s=3, n=3, m=0$ coincides with the quartic determinantal K3 surface studied by Iliev and Manivel in \cite{iliev2019hyperkaehler}. Letting $s,n,m$ vary, we find many examples of interesting varieties, including surfaces of general type.

Our idea is to study $\Hilb^2(S)$ via an auxilary hypersurface $Y \subset \PP^s \times \PP^{n+m-1} \times \PP^{n-1}$ naturally associated to $S$, and defined explicitly in \Cref{eqn:Ynsm}. The variety $Y$ is always a \emph{Fano variety}, whose study was one of the initial motivations for our project. Via a modular-type construction we then pass from $Y$ to $Z$, defined as
\[
Z = V(\omega) \hookrightarrow \Gr(2, n)\times\Gr(2, s+1) \times \Gr(n+m-2, n+m),
\]
where $\omega$ is a tri-tensor naturally attached to $\varphi$. As explained in \Cref{subsection.degeneracy}, $Z$ is the zero locus of a section of an irreducible, globally generated, homogeneous vector bundle naturally associated to $Y$.

Our main result proves that, in a certain infinite range, the variety $Z$ and the Hilbert square $\Hilb^2(S)$ of the variety we started with are isomorphic. Namely, we have the following.

\begin{thm}[\Cref{thm.Hilbiso}]\label{thm:mainthm_intro}
Let $n\geq3$, $m\geq0$, $s \in \set{m+2\,,\dots,\,2m+3}$. Assume $n > 2s-2m-3$. Then, there is an isomorphism of schemes $\vartheta \colon Z \simto \Hilb^2(S)$.
\end{thm}

Our proof goes via the explicit construction of the morphism $\vartheta$. In principle, it says nothing on the cases $n \leq 2s-2m-3$. However, we show that for low values of $m$ these two varieties are not even deformation equivalent --- indeed, their topological Euler characteristics are different. This observation leads us to conjecture that, in fact, our bound is optimal, see \Cref{conj.optimalbound}.

In Sections~\ref{section.setup}--\ref{sec:ex} we explain the geometric setup and the main motivating ideas behind this paper; we also explicitly describe various examples in which our result applies. Sections~\ref{section.linearalgebra}--\ref{section.badlines} are the technical core of this paper: first we describe in full detail the geometry of $Z$ (cf.~\Cref{thm.omegazerolocus}) independently upon the choice of $s,n,m$, then we explain how the cases in which our main result does \emph{not} work are related to the presence of some \emph{special lines} contained in $S$ (cf.~\Cref{thm.badlines} and \Cref{thm.verybadlines}). Our main result, \Cref{thm:mainthm_intro}, is proved in \Cref{section.mapZHilb} (cf.~\Cref{thm.Hilbiso}), whereas \Cref{section.examples} is devoted to the study of the geometry of some interesting varieties arising as the limit cases for which our method fails, but enjoying a beautiful and rich geometry. Among these examples we include \emph{generalised Bordiga scrolls} (cf.~\Cref{ex.bordigascroll}), \emph{higher dimensional White varieties} (cf.~\Cref{ex.whitevarieties}), and also certain varieties containing a finite number of special lines (cf.~\Cref{example.3foldm1} and \Cref{conj.optimalbound}).

\subsection*{Notation}
We work over the field of complex numbers $\mathbb{C}$.
For an arbitrary positive integer $d$ we let $V_d$ be a $d$-dimensional $\mathbb{C}$-vector space, which we also identify with $d$-dimensional affine space $\mathbb A^d$.

We denote by $\Gr(k,n)$ the Grassmannian of $k$-dimensional subspaces in $V_n$. We denote with $\mU$ the rank $k$ tautological vector bundle over it, with anti-ample determinant. We write $X=(G,\mF)$ to denote the zero locus $X=V(\sigma) \subset G$, for a general section $\sigma \in \mathrm{H}^0(G, \mF)$ of a vector bundle $\mF$ on a variety $G$. Sometimes we will need to work with a specific $\sigma$, and we will specify it accordingly.

\subsection*{Acknowledgements}
We are grateful to Kieran O'Grady and Claudio Onorati for useful discussions on the subject of this paper. The first three authors are members of INDAM-GNSAGA. The authors have been partially supported by PRIN2017 2017YRA3LK and PRIN2020 2020KKWT53.

\section{Setup, motivation and some toy cases}\label{section.setup}

\subsection{Degeneracy loci, Fano varieties and Hilbert schemes}\label{subsection.degeneracy}
We start by considering a very simple construction from linear algebra. We consider a general $n \times(n+m)$ matrix of homogeneous linear forms
\[
M=\begin{pmatrix}
f_1^1 & \ldots & f_{n+m}^1 \\
f_1^2 &\ldots & f_{n+m}^2 \\
\vdots& \ddots & \vdots \\
f_1^n & \ldots & f_{n+m}^n \\
\end{pmatrix}
\]
on an ambient projective space $\PP^s=\PP(V_{s+1})$. If we ask for $M$ to have non-maximal rank, we have to consider the locus where all the $n+1$ maximal minors vanish. This is of course equivalent to the existence of some linear relations between the rows of $M$.
We can therefore consider two strictly related loci: the first one is $S_{n,s,m} \subset \PP^s$,
 given by the vanishing of the maximal minors of $M$ --- i.e.~the \emph{first degeneracy locus} $D_{n-1}(\varphi)$ --- where we implicitly identify the matrix with the morphism $\varphi \colon \mathcal O_{\mathbb P^s}^{\oplus n+m} \to \mathcal O_{\mathbb P^s}(1)^{\oplus n}$ defining it. Sometimes we will shorten $S_{n,s,m}$ with $S$, when the subscripts are clear from the context.
 In other words,
 \[
 S_{n,s,m} = \Set{[v] \in \PP^s \ | \rank (M_v) \leq n-1},
 \]
where $M_v \in \mathsf{Mat}_{n,n+m}(\mathbb C)$ is the evaluation of $M$ at $v \in V_{s+1}$. 

\smallbreak
The second relevant locus is a subvariety $X_{n,s,m}\subset \PP^s \times \PP^{n+m-1}$, given by $n$ bihomogeneous linear polynomials of bi-degree $(1,1)$, i.e.~by a section of $\of(1, 1)^{\oplus n}$.
 
 What is the relation between $S$ and $X$? First of all, assume $S$ to be smooth with $\dim(S)>0$. Under our generality assumption, this will be equivalent to requiring $m+2 \leq s \leq 2m+3$, where the second inequality ensures that the further degeneracy loci will be empty.

Now, $X_{n,s,m}$ is constructed in a tautological way as follows: if $y_1, \ldots, y_{n+m}$ are chosen coordinates on $\PP^{n+m-1}$, and  $F_i= (f_1^i, \ldots, f_{n+m}^i)$ is the $i$-th row of our matrix $M$, we will have
\[
X_{n,s,m} = V(F_1 \cdot \underline{y}, \ldots, F_n \cdot \underline{y}) \subset \PP^s \times \PP^{n+m-1}.
\]
Using our notation,
\begin{equation}\label{eqn:X_nsm}
X_{n,s,m}= \left(\PP^s \times \PP^{n+m-1}, \of(1,1)^{\oplus n}\right).
\end{equation}
The fibres of the projection $\pi\colon X_{n,s,m} \hookrightarrow \PP^s \times \PP^{n+m-1} \to \PP^s$ are generically cut out by $n$ linear equations,
or $n-1$ exactly where there is a linear dependence relation in $M$ (and that is all that can happen, since by hypothesis there are no further degenerations): in other words, we have proved the following lemma.

\begin{lemma}
In the setup above, $\pi\colon  X_{n,s,m} \to \PP^s$ is generically a $\PP^{m-1}$-bundle jumping to a $\PP^{m}$-bundle exactly over $S_{n,s,m}$.
\end{lemma}

We call $X=X_{n,s,m}$ a generalised $(m-1, m)$ blow-up of $S=S_{n,s,m}$.
This construction is sometimes referred to as \emph{Cayley trick}. This is in fact a generalisation of the blow-up formula, and it implies that the vanishing cohomologies of $X$ and $S$ are isomorphic, and also that $D^b(X)$, the bounded derived category of coherent sheaves over $X$, contains a copy of $D^b(S)$. References for this fact can be found in \cite[Theorem 2.4]{kkll} and \cite[Proposition 46]{bfm}.

\smallbreak
We could have built yet another natural variety starting from the matrix $M$ (or better, its transpose). If we take the transpose $M^t$ of the matrix $M$, and we apply it to a vector $\underline{z}=(z_1, \ldots, z_n)^t$ we can consider the locus $\Gamma_{n,s,m} \subset \PP^s \times \PP^{n-1}$, given by $M^t \cdot \underline{z} = 0$. In other words, if we write $F^t_i=(f_i^1, \ldots, f_i^n)$, we have then
\[
\Gamma_{n,s,m}=V(F^t_1 \cdot \underline{z}, \ldots, F^t_{n+m} \cdot \underline{z}) \subset \PP^s \times \PP^{n-1}
\]
and again, in our notation,
\[
\Gamma_{n,s,m}= (\PP^s \times \PP^{n-1}, \of(1,1)^{\oplus n+m}).
\]
Consider, this time, the restricted projection $\Gamma_{n,s,m} \hookrightarrow \PP^s \times \PP^{n-1} \to \PP^s$. This time the fibre is generically empty, and it becomes a point exactly where the rank drops, i.e.~on $S$. In other words, one has the following lemma.

\begin{lemma}\label{lem:proj}
The projection $\PP^s \times \PP^{n-1} \to \PP^s$ restricts to an isomorphism $\Gamma_{n,s,m} \simto S_{n,s,m}$.
\end{lemma}

This implies that the Picard group of $S_{n,s,m}$ is $\Z^2$ (at least generically), and the line bundles $\of(1,1)$ and $\of(1,0)$ (restricted from $\PP^s \times \PP^{n-1}$) are both very ample.
In what follows, we will study as well the morphism induced by $\of(0,1)$, showing that it will be very ample in a certain range (namely $n>2s-2m-3$) as well.

\smallbreak
Consider now two triples $(n_1, s_1, m_1)$ and $(n_2, s_2, m_2)$: if we set $n_2=s_1+1, s_2=n_1-1, m_2=n_1+m_1-s_1-1$, then $\Gamma_{n_1, s_1, m_1}$ and $\Gamma_{n_2, s_2, m_2}$ are both $(n_1+m_1)$-codimensional linear sections of $\PP^{s_1} \times \PP^{n_1-1}$, with the role of the two projective spaces exchanged, hence they belong to the same deformation family. When the triples satisfy such a relation, we call them \emph{associated}.

If we are in the correct range for the first triple, i.e.~$m_1+2 \leq s_1 \leq 2m_1+3$, $n_1\geq3$ and $n_1>2s_1-2m_1-3$, then the second triple will be in the correct range as well (in fact $n_2>2s_2-2m_2-3$ reduces exactly to $s_1 \leq 2m_1+3$).

In this range both projections to $\PP^{s_1}$ and $\PP^{n_1-1}$ are embeddings when restricted to $\Gamma$ (this follows from \Cref{thm.badlines}): in other words, \[
S_{n_2, s_2, m_2}\cong S_{s_1+1, n_1-1,n_1+m_1-s_1-1}
\]
yields another presentation for $S_{n_1, s_1, m_1}$, with a different embedding. We will see these phenomena in detail when dealing with two presentations of  determinantal quartic K3 surfaces (abstractly but not projectively isomorphic), and of a quintic determinantal surface embedded as codimension 2 degeneracy locus, see \Cref{subsec:toy_I}.

\smallbreak
Let us now get back to $X=X_{n,s,m}$, and perform once again a \emph{Cayley trick}. In fact, we can associate to $X_{n,s,m}$ another variety
\begin{equation}\label{eqn:Ynsm}
Y_{n,s,m} =(\PP^s \times \PP^{n+m-1} \times \PP^{n-1}, \of(1,1,1)),
\end{equation}
defined tautologically starting from the equations of $X$. This will be simply given by
\[
Y_{n,s,m}= V\left(\sum_{i=1}^n z_i (F_i \cdot \underline{y}) \right).
\]
Of course, the projection $\PP^{n-1}\times\PP^s \times \PP^{n+m-1} \to \PP^s \times \PP^{n+m-1}$ restricted to $Y$ is generically a $\PP^{n-2}$-bundle, with special fibres the whole $\PP^{n-1}$ over  $X$.

Notice that $Y=Y_{n,s,m}$ is a Fano variety, simply by adjunction: on the other hand this is not the case in general for $X$ or $S$: as a matter of fact, we will work only under certain (at least) non-negativity assumption for the canonical bundle of $S$. 

In a certain sense, the main character of the whole story is precisely the Fano variety $Y$: we can see it as the universal variety associated to a tri-tensor $\omega \in V_s^{\vee} \otimes V_{n+m}^{\vee} \otimes V_n^{\vee}$ simply given by $\omega= \sum_{1\leq i\leq n} z_i (F_i \cdot \underline{y})$. To be precise, we should have a dual in the last component --- equivalently, it would be more natural to have the dual $(\PP^{n-1})^{\vee}$- but we silently use the duality isomorphism to unburden the notation.

The geometry of a tri-tensor is an old and fascinating topic, with one of the first references being \cite{cayley1869memoir}. See also, \cite{Ottaviani_1, tanturri2013degeneracy} for a modern account. The degeneracy locus $S$, the rational variety $X$ and all the other characters appearing in this picture can be seen to be induced by $Y$ via the obvious projections.

Finally, we associate to $Y_{n,s,m}$ one last variety $Z_{n,s,m}$, which is far from being a Fano variety. Denote by
\begin{equation}\label{def:triple_grass}
G_{n,s,m} \defeq \Gr(2, n)\times\Gr(2, s+1) \times \Gr(n+m-2, n+m),
\end{equation}
then define the vanishing locus $Z_{n,s,m}=V(\omega)\subset G_{n,s,m}$. In our notation,
\begin{equation}\label{eq:z}
Z_{n,s,m}= (G_{n,s,m} \,,\, \mU^{\vee} \boxtimes \mU^{\vee} \boxtimes \mU^{\vee}).
\end{equation}
The reason for this apparently arbitrary choice is that by Borel--Bott--Weil 
\[ 
\mathrm{H}^0(\PP^{n-1}\times \PP^s \times \PP^{n+m-1}, \of(1,1,1) ) \cong \mathrm{H}^0 \left(G_{n,s,m}\,,\, \mU^{\vee} \boxtimes \mU^{\vee} \boxtimes \mU^{\vee}\right).
\]
Notice that this holds true for any product $\Gr(k_3, n) \times \Gr(k_1, s+1) \times \Gr(k_2, n+m)$. However, with this particular choice of ambient spaces, we have that the dimension of $Z$ is equal to $2(s-m-1)$, i.e.~$\dim Z = 2\cdot\dim S$.

This is not a coincidence: in fact the purpose of this paper is to show that as long as the triple $(n,s,m)$ satisfies the constraints 
\[
m+2 \leq s \leq 2m+3, \quad n >2s-2m-3,
\]
one has an isomorphism of schemes
\[
Z_{n,s,m} \cong \Hilb^2(S_{n,s,m}).
\]

We stress that the condition $n>2s-2m-3$ is not an if and only if. In fact our proof goes via the explicit construction of a morphism to the Hilbert scheme, which exists and happens to be an isomorphism in that range. This a priori says nothing on the other cases. However, we show that for, e.g.~$m=0,1$ our bound is optimal, see \Cref{ex.whitesurfaces} and \Cref{example.3foldm1} where we explicitly compute the Hodge numbers of $Z$ and $\Hilb^2(S)$ in the range $n \leq 2s-2m-3$, thus confirming that they are different.

\subsection{A conjectural relation with the Hilbert scheme of the Fano variety \texorpdfstring{$Y$}{}}
Before discussing some examples, we mention one more relation between $Z$ and the Hilbert scheme, that we leave for future research to explore. More precisely, we conjecture that $Z$ can be realised as a Hilbert scheme \emph{on $Y$} as well. In fact, if we call $\PP_{1,1,n-3}  \defeq \PP^1 \times \PP^1 \times \PP^{n-3}$ contained fibre-wise in $\PP_{n,s,m} \defeq \PP^{n-1}\times \PP^s \times \PP^{n+m-1}$, we can consider the incidence variety
\[
F = \set{(p, \PP_{1,1,n-3}) \subset \PP_{n,s,m} \times  G_{n,s,m} | \ p \in \PP_{1,1,n-3}},
\]
with $G_{n,s,m}$ as in \eqref{def:triple_grass}.

Notice that $F$ can be described as the zero locus
\[ F= (\Fl(1,2,n) \times \Fl(1,2,s+1) \times \Fl(n+m-3, n+m-2, n+m), \of(1,0) \otimes \of(1,0) \otimes \mQ_2) \; , \]
where the first two bundles are the pullback of the ample line bundles from $\PP^{n-1}$ and $\PP^s$, and the last is the pullback of the rank 2 quotient bundle in $\Gr(n+m-2, n+m)$. This implies that the projection $p$ from $F$ to $Y$ is a $\PP^{n-2} \times \PP^{s-1} \times \PP^{n+m-2}$-bundle, while the projection $q$ from $F$ to $G_{n,s,m}$ is a $\PP^{n+m-3} \cup \PP^{n+m-3} $ generically, degenerating to a $\PP^{n+m-3} $ over $Z$. We believe that one should also have an isomorphism
\[Z 
\cong \Hilb_{\PP_{1,1,n-3}}(Y)
\]
with the induced isomorphism in cohomology realised by the classical \emph{Abel-Jacobi type} $p^*q_*$-map. However, we have not been able to prove this for the time being, and we hope to return to it in the future.

\subsection{Toy case I: determinantal}\label{subsec:toy_I}
As a first special sub-case, it is worth mentioning the case $m=0$, in which case $S$ is a determinantal hypersurface in $\PP^s$. Also, we need $s \leq 3$, for from threefolds onwards $S$ will in fact be singular.

With $s=3$, the last case excluded by our theorem, $n=3$, is the one of a cubic surface, and we can immediately show that $Z$ and $\Hilb^2(S)$ are not isomorphic: as a matter of fact, $e_{\mathrm{top}}(\Hilb^2(S)) = e_{\mathrm{top}}(Z)+21$, where the discrepancy by $21$ should be accounted for by the 6 exceptional lines plus the other 15 which are strict transforms of lines passing through two of the six points. 

If we consider $n=4$, we have that $X \cong \Gamma \cong S$, and with three different representations. In this case the isomorphism was already known to be true from \cite[Proposition 1]{iliev2019hyperkaehler}. In fact in this case $S$ is a determinantal quartic K3 surface, presented with three different models, hence $Z$ is a hyperk\"ahler fourfold. This construction is very classical, starting from \cite{cayley1869memoir}, and the relations between the three models has been recently explored in \cite{festi2013cayley, oguiso2017isomorphic, veniani2019symmetries}. 

Another interesting case which is covered by our theorem is the one of a determinantal \emph{quintic surface}, which we will explore in detail in \Cref{sec:45}.

\subsection{Toy case II: sub-determinantal}\label{subsec:toy_II}
Another relevant case is the sub-determinantal case, i.e.~for $m=1$. In this case we can borrow some results from \cite[\S 2.2]{kuznetsovc5} and \cite[Proposition 3.6]{bfmt} to readily compute the invariants of $S$. We remark that our smoothness condition forces $3 \leq s \leq 5$.
In fact, the $k$-th degeneracy locus $D_{n-k}(\varphi)$ has expected codimension $k(m+k)$ in the ambient space $\PP^s$. Tence for $m=1$, $k=2$, it has expected codimension 6, i.e.~$D_{n-2}(\varphi)=\emptyset$.

Notice how in this case the map $X \to \PP^s$ is particularly simple, indeed it agrees with the blow up map $X= \Bl_S \PP^s \to \PP^s$.

The structure sheaf of $S=D_{n-1}(\varphi)$ admits a resolution by the so-called Eagon--Northcott complex. In this case, it takes the form: 
\begin{equation}\label{eq:e-n}
    0 \to \mF^{\vee} \to \mE^{\vee} \to \det (\mE^{\vee}) \otimes \det (\mF) \to (\det (\mE^{\vee}) \otimes \det (\mF))|_{D_{n-1}(\varphi)} \to 0.
\end{equation}
This holds more in general for every $\mE, \mF$ of rank $(n+1,n)$: in our case it will suffices to take $\mE \cong \of_{\PP^s}^{\oplus n+1}$ and $\mF\cong \of_{\PP^s}(1)^{\oplus n}$.

One can use suitably twisted versions of this complex to compute some invariants of $S$, as shown in the next examples. Of course one could have worked directly on $\Gamma$ as well, or on $X$, applying the blow-up formula.

\begin{proposition}\label{prop.genus} 
Fix $s=3$ and $n>1$. Let $\varphi\colon \mathcal O_{\PP^3}^{\oplus n+1} \to \mathcal O_{\PP^3}(1)^{\oplus n}$ be a general morphism of vector bundles. Consider the smooth curve $S_n\defeq S_{n,3,1} = D_{n-1}(\varphi) \subset \PP^3$. Then 
\[
g(S_n)=n\binom{n}{3} - (n+1) \binom{n-1}{3}, \qquad \deg(S)= \binom{n+1}{2}.
\]
\end{proposition}
\begin{proof}
Consider the Eagon--Northcott resolution of $\of_{S_n}$ from \eqref{eq:e-n}. Twisting back, we have 
\[
0 \to \of_{\PP^3}(-n-1)^{\oplus n } \to \of_{\PP^3}(-n)^{\oplus n+1} \to \of_{\PP^3} \to \of_{S_n} \to 0.
\]
We have that \[ 
\chi(\of_{S_n}) = \chi(\of_{\PP^3}) + \chi(\of_{\PP^3}(-n-1)^{\oplus n } ) - \chi(\of_{\PP^3}(-n)^{\oplus n+1}),
\]
i.e.
\[
1-g(S_n)= 1-n\binom{n}{3} + (n+1) \binom{n-1}{3}.
\]
Therefore, we only need to check that $S$ is connected. This can be done by splitting \eqref{eq:e-n} in two short exact sequences \[0 \to \of_{\PP^3}(-n-1)^{\oplus n } \to \of_{\PP^3}(-n)^{\oplus n+1} \to K \to 0
\]
\[
0 \to K \to \of_{\PP^3} \to \of_{S_n} \to 0.
\]
Since $K$ has no cohomologies in $h^0, h^1$, it follows that $h^0(\of_{S_n}) = h^0(\of_{\PP^3})=1$.

In order to compute the degree, it suffices to check the Hilbert polynomial, which for a curve we know to be equal to $p_{S_n}(t)=dt+1-g$, where $d$ is the degree. Since in general $p_{S_n}(t)=at+b$, we have of course $\chi(\of_{S_n})=p_{S_n}(0)=1-g$ and \[\chi(\of_{S_n}(1))=p_{S_n}(1)= 4- n\binom{n-1}{3}+ (n+1) \binom{n-2}{3},\] where we used as before the sequence \eqref{eq:e-n}. It follows that \[a= 3+ n \left(\binom{n}{3}- \binom{n-1}{3}\right) - (n+1) \left(\binom{n-1}{3}-\binom{n-2}{3}\right),
\]
which simplifies to $a=\binom{n+1}{2}$. The result follows.
\end{proof}

\begin{proposition}\label{prop:euler_char}
Fix $s \in \{3,4\}$ and $n>1$. Let $\varphi\colon \mathcal O_{\PP^s}^{\oplus n+1} \to \mathcal O_{\PP^s}(1)^{\oplus n}$ be a general morphism of vector bundles. Then the smooth subvariety $S_{s,n}\defeq S_{n,s,1} = D_{n-1}(\varphi) \subset \PP^s$, of codimension $2$, has topological Euler characteristic
\[
e_{\mathrm{top}}(S_{s,n})=
\begin{cases}
4n^2-2n^3+(3n-4)\binom{n}{2}-\binom{n}{3} & \textrm{if }s=3 \\
n^2(10-10n+3n^2)+\binom{n}{2}(-10+15n-6n^2)+\binom{n}{3}(4n-5)-\binom{n}{4} & \textrm{if }s=4
\end{cases}
\]
\end{proposition}

\begin{proof}
See \Cref{sec:Euler}.
\end{proof}

\begin{lemma} 
Fix $s = 4$. Then the smooth surface $S_{n,4,1} \subset \PP^4$ has irregularity $q=0$, and geometric genus
\[
p_g(S_{n,4,1}) = n \binom{n}{4} - (n+1) \binom{n-1}{4}.
\]
\end{lemma}

\begin{proof}
The Euler characteristic of the structure sheaf $\chi(\of_{S_n})$ is computed as in the previous proposition, using the sequence \eqref{eq:e-n} on $\PP^4$. We have in particular that \[
\chi(\of_{S_n})= 1 + n \binom{n}{4} - (n+1) \binom{n-1}{4}.
\]
Moreover, $S_n$ is connected and $q=0$. The first statement can be proven as in the curve case. The second follows from the isomorphism $\Gamma \cong S_n$. On the other hand, we know that $\Gamma = (\PP^4 \times \PP^{n-1}, \of(1,1)^{\oplus n+1})$. Hence, by Lefschetz hyperplane section theorem, the only weight where the cohomology of $\Gamma$ has non-zero level is the middle one; therefore, $q=0$.
\end{proof}

\begin{remark}
From the above lemma one immediately deduces that $p_g=q=0$ as long as $n<4$. Moreover the same argument tells us that for a threefold which is a degeneracy loci in $\PP^5$, $h^1(\of_{S_n})=h^2(\of_{S_n})=0$ and $h^{1,1}(S_n)=2$.
\end{remark}

A nice observation is that the sub-determinantal case $n=4,s=4, m=1$ and the determinantal case $n=5, s=3, m=0$ both give rise to a determinantal quintic, since $\Gamma$ in both cases is given by
\[
\Gamma=(\PP^3 \times \PP^4, \of(1,1)^{\oplus 5}),
\]
albeit the role of $\PP^3$ and $\PP^4$ is exchanged.

\section{Some examples}\label{sec:ex}
In this section, we collect some examples that do fall within the `good range' prescribed by \Cref{{thm:mainthm_intro}}, and that therefore realise the desired isomorphism $\Hilb^2(S) \cong Z$. For the sake of completeness, we write down the Hodge numbers of the varieties involved, which can be computed using the methods detailed in \cite[\S 3.2]{dft}.

\subsection{The cases \texorpdfstring{$n=3, s=3, m=1$}{} and \texorpdfstring{$n=4, s=2, m=0$}{}}
We discuss first an example which is quite classical. Let us consider $S \subset \PP^3$, where $S$ is a degree 6, genus 3 space curve given by the intersection of four cubics (i.e.~the maximal minors of a $4 \times 3$ matrix of linear forms).

In the notation of the previous section, according to \eqref{eqn:X_nsm} in the case $(n,s,m)=(3,3,1)$ we have $X \subset \PP^3 \times \PP^3$, given as the complete intersection of three divisors of bi-degree $(1,1)$, i.e.~$X=(\PP^3 \times \PP^3, \of(1,1)^{\oplus 3})$.
This variety $X$ is the Fano 3-fold \textbf{2--12} in the original Mori--Mukai notation, see \cite{fanography, dft, morimukai}.

Following the discussion of the previous section, $X$ is identified with the blow-up $\Bl_{S} \PP^3$, see also \cite[2-12]{ccgk}.
One can immediately compute the Hodge numbers of $X$, these being
\begin{center}
{\small
\[\begin{matrix}

&&& 0 && 3 && 3 &&0&&& \\
&&&&0 &&2&&0&&&&\\
&&&&&0&&0&&&&&\\
&&&&&&1 &&&&&&
\end{matrix}\]}
\end{center}
 The rational map $\eta: \PP^3 \dashrightarrow \PP^3$ induced by this construction is the the cubo-cubic Cremona transformation of $\PP^3$ already known to Max Noether, see \cite{noether1871ueber,reede2019cubo} and it is the only non- trivial Cremona transformation of $\PP^3$ that is resolved by just one blow up along a smooth curve, see \cite{katz1987cubo}.

The second variety in the picture is $Y= (\PP^3 \times \PP^3 \times \PP^2, \of(1,1,1))$. This is a Fano 7-fold, with anti-canonical class equal to $-K_Y \cong \of_Y(3,3,2)$. We can apply the Cayley trick from $Y$ to $X$ to determine the Hodge numbers of $Y$, which can be also computed using the standard Koszul resolution. These are:
\begin{center}
{\small
\[\begin{matrix}
0 && 0&& 0 && 3 && 3 && 0 && 0 &&0 \\
&0&&0&&0 &&9 &&0 &&0&&0&\\
&&0&&0&&0&&0&&0&&0 && \\
&&&0&&0&&6&&0&&0&&&\\
&&&& 0 && 0 && 0 &&0&&&& \\
&&&&&0 &&3&&0&&&&&\\
&&&&&&0&&0&&&&&&\\
&&&&&&&1 &&&&&&&
\end{matrix}\]}
\end{center}

Finally, we consider the variety $Z=(\Gr(2,4)\times \Gr(2,4) \times \Gr(2,3), \mU^{\vee} \boxtimes \mU^{\vee} \boxtimes \mU^{\vee} )$. By our theorem, $Z \cong \Hilb^2(S) \cong \Sym^2(S)$. We can check that $K_Z \cong \of_Z(0,0,1)$ and that its Hodge numbers are the expected ones, namely

\begin{center}
{\small
\[\begin{matrix}

&&&&3 &&10&&3&&&&\\
&&&&&3&&3&&&&&\\
&&&&&&1 &&&&&&
\end{matrix}\]}
\end{center}

Finally, notice that, using the notation of the previous section, the associated triple to $(3,3,1)$ is $(4,2,0)$. In this case $S_{4,2,0} \subset \PP^2$ is a plane quartic curve, and $Z$ describes its symmetric square as well.

\subsection{The case \texorpdfstring{$n=5, s=3, m=0$}{} and \texorpdfstring{$n=4, s=4,  m=1$}{}}\label{sec:45}
As before, these two cases define the same surface, in two different presentations. In fact, the first triple of invariants immediately identifies $S_{5,3,0} \subset \PP^3$ as a quintic determinantal surface, which has Picard rank 2 in general. On the other hand $S_{4,4,1} \subset \PP^4$ is a codimension 2 surface defined by 5 quartic equations. However, thanks to \Cref{lem:proj} they are both isomorphic to the same $\Gamma$, which is $$\Gamma= (\PP^3 \times \PP^4, \of(1,1)^{\oplus 5}).$$
 The Hodge numbers of $S$ (of course regardless of the presentation) are as follows:

\begin{center}
{\small
\[\begin{matrix}

&&&&4 &&45&&4&&&&\\
&&&&&0&&0&&&&&\\
&&&&&&1 &&&&&&
\end{matrix}\]}
\end{center}

We can consider the associated $Z =(\Gr(2,5) \times \Gr(3,5) \times \Gr(2,4), \mU^{\vee}\boxtimes \mU^{\vee} \boxtimes \mU^{\vee})$, which is of course the same in both cases. The Hodge numbers of $Z \cong \Hilb^2(S)$ (see also \Cref{sec:hodge-poly}) are:

\begin{center}
{\small
\[\begin{matrix}
&&&10&&184&&1097&&184&&10&&&\\
&&&& 0 && 0 && 0 &&0&&&& \\
&&&&&4 &&46&&4&&&&&\\
&&&&&&0&&0&&&&&&\\
&&&&&&&1 &&&&&&&
\end{matrix}\]}
\end{center}

\subsection{The case \texorpdfstring{$n=6, s=5, m=1$}{}}\label{sec:651}

If for $m\in\set{0,1}$ in the surface case our condition $n>2s-2m-3$ corresponded essentially to a non-negative Kodaira dimension, for $(m,s)=(1,5)$, the limit case which is not covered by \Cref{thm:mainthm_intro}, is a threefold of general type: in fact, we are going to show in \Cref{example.3foldm1} that $Z$ and $\Hilb^2(S)$ are \emph{not} isomorphic. In fact, the first case with $m=1, s=5$ which is covered by our Theorem is for $n=6$. In this case, the associated triple to $(6,5,1)$ is again $(6,5,1)$.

Our threefold $S_{6,5,1} \subset \PP^5$ is defined by 7 minors (of degree 6): it is isomorphic to $\Gamma =( \PP^5 \times \PP^5, \of(1,1)^{\oplus 7})$.

We can compute the Hodge numbers of $S_{5,6,1}$, these being:

\begin{center}
{\small
\[\begin{matrix}
&&&& 29 && 520 && 520 &&29&&&& \\
&&&&&0 &&2&&0&&&&&\\
&&&&&&0&&0&&&&&&\\
&&&&&&&1 &&&&&&&
\end{matrix}\]}
\end{center}

The Hodge numbers of $\Hilb^2(S) \cong Z \subset \Gr(2,6) \times \Gr(2,6) \times \Gr(5,7)$ are

\begin{center}
{\tiny
\[\begin{matrix}

&406&&15080&&150020 &&271250 &&150020 &&15080&&406&\\
&&0&&87&&1560&&1560&&87&&0 && \\
&&&0&&0&&8&&0&&0&&&\\
&&&& 29 && 520 && 520 &&29&&&& \\
&&&&&0 &&3&&0&&&&&\\
&&&&&&0&&0&&&&&&\\
&&&&&&&1 &&&&&&&
\end{matrix}\]}
\end{center}
Notice that the Euler characteristic $e_{\mathrm{top}}(Z) = 593502$ coincides with $e_{\mathrm{top}}(\Hilb^2(S))$, which is computed in \cref{sec:hodge_3fold}.

\subsection{The cases \texorpdfstring{$n=4, s=5, m=2$}{} and \texorpdfstring{$n=6, s=3, m=0$}{}}

These two associated cases describe two different presentation for $S$, as a smooth determinantal sextic and as a codimension 3 surface in $\PP^5$. Here $\Gamma$ can be described as $\Gamma= (\PP^5 \times \PP^3, \of(1,1)^{\oplus 6})$. The Hodge numbers for $S$ are:
\begin{center}
{\small
\[\begin{matrix}
&&&&10 &&86&&10&&&&\\
&&&&&0&&0&&&&&\\
&&&&&&1 &&&&&&
\end{matrix}\]}
\end{center}

We can compute the Hodge numbers of $\Hilb^2(S) \cong Z \subset \Gr(2,6) \times \Gr(2,4) \times \Gr(4,6)$, which are:
\begin{center}
{\small
\[\begin{matrix}
&&&55&&870&&3928&&870&&55&&&\\
&&&& 0 && 0 && 0 &&0&&&& \\
&&&&&10 &&87&&10&&&&&\\
&&&&&&0&&0&&&&&&\\
&&&&&&&1 &&&&&&&
\end{matrix}\]}
\end{center}

\section{Key construction and preparation lemmas}\label{section.linearalgebra}
In this section we explain the key constructions that will allow us to prove \Cref{thm.Hilbiso}.

We fix integers $n\geq 3$, $m\geq 0$, and $s\in\{m+2\, ,\dots,\, 2m+3\}$.

\subsection{Constructing points in the triple Grassmannian}
We shall consider a \emph{general} map of vector bundles
\begin{equation}\label{eqn:bundlemap}
\varphi \colon \of_{\mathbb P^s}^{\oplus n+m} \to \of_{\mathbb P^s}(1)^{\oplus n}
\end{equation}
along with the associated  $(m+1)$-codimensional, \emph{smooth} degeneracy locus 
\[
S = S_{n,s,m} = D_{n-1}(\varphi) \hookrightarrow \mathbb P^s.
\]
Indeed, by the genericity of $\varphi$, each degeneracy locus $D_k(\varphi) = \{p \in \mathbb P^s | \rank(\varphi(p))\leq k\} \subset \mathbb P^s$ has codimension in $\PP^s$ equal to the expected one, namely $(n-k)(n+m-k)$. In the range $s\in \{m+2\,,\ldots,\,2m+3\}$, we have $\dim D_{n-1}(\varphi)>0$, and the singularities may only arise in $D_{n-2}(\varphi)=\emptyset$, whence the smoothness.

Equivalently, $\varphi$ can be understood from an algebraic point of view as a matrix
\[ 
M=\begin{pmatrix}
f_1^1 & \ldots & f_{n+m}^1 \\
f_1^2 &\ldots & f_{n+m}^2 \\
\vdots& \ddots & \vdots \\
f_1^n & \ldots & f_{n+m}^n \\
\end{pmatrix}\,\in\,\mathsf{Mat}_{n,n+m}\left(\mathrm{H}^0(\mathbb P^s,\mathcal{O}_{\mathbb P^s}(1))\right)
\]
of linear forms $f^i_j$ depending on $s+1$ variables. We shall switch from $\varphi$ to $M$ freely in what follows. 

Working in the affine setup, one is led to consider the locus 
\[
\widehat{S}=\Set{v\in V_{s+1}\,\vert\, \rank(M_v)=n-1}\subset V_{s+1},
\]
where $M_v \in \mathsf{Mat}_{n,n+m}(\mathbb C)$ denotes the matrix $M$ evaluated at the point $v \in V_{s+1}$. By the linearity of $f^{i}_{j}$, the subvariety $\widehat{S}\subset V_{s+1}$ descends to a subvariety $S\hookrightarrow \mathbb{P}^s=\mathbb{P}(V_{s+1})$, in such a way that $\widehat{S}\cup \set{0}$ is the affine cone over $S\hookrightarrow \mathbb{P}^s$. We shall use the notation $[v]$ to denote a point in projective space, to emphasise that we take the projective point of view.

Consider the set-theoretic map
\[ 
\psi\colon S\to \mathbb{P}^{n-1}, \qquad [v]\mapsto [\alpha_v],
\]
where $[\alpha_v]$ is determined by the $1$-dimensional $\mathbb C$-vector space
\[
\ker\left(\of_{\mathbb P^s}(-1)\big|_{[v]}^{\oplus n} \xrightarrow{\varphi_{[v]}^t} \of_{\mathbb P^s}\big|_{[v]}^{\oplus n+m} \right) \,\subset\,\of_{\mathbb P^s}(-1)\big|_{[v]}^{\oplus n} = V_n = \mathbb C^n.
\]
Of course, if $M$ is the $n\times (n+m)$ matrix of linear forms corresponding to the morphism $\varphi$, then $\alpha_v\in V_n$ is defined (up to scalar multiplication) by $M_v^t\cdot\alpha_v=0$.

\begin{lemma}\label{lemma.psi}
The association $[v]\mapsto [\alpha_v]$ defines an algebraic morphism $\psi\colon S \to \mathbb P^{n-1}$.
\begin{proof}
As already mentioned, since $\varphi$ is general, one has $D_{n-2}(\varphi)=\emptyset$, and thus $\varphi_{[v]}$ has rank precisely $n-1$ for every $[v] \in S$. Therefore the sheaf $\mathcal L=\textrm{coker}(\varphi)|_S$ is a locally free sheaf of rank $1$, and moreover it is globally generated by $n$ sections $\alpha^1,\ldots,\alpha^n$, arising from the linear dependence relations
\[
\alpha^1_{[v]}F^1_{[v]} + \cdots + \alpha^n_{[v]}F^n_{[v]}= 0,\quad [v] \in S,
\]
where $F^i_{[v]}$ denotes the $i$-th row of the matrix associated to $\varphi_{[v]} = M_v$. The data $(\mathcal L,\alpha^1,\ldots,\alpha^n)$ defines the sought after algebraic morphism $\psi\colon S \to \mathbb P^{n-1}$.
\end{proof}
\end{lemma}

Our key construction starts now.
Let $[v],[w]\in S$ be distinct points and consider the space
\begin{equation}\label{eqn:crucial_plane}
\pi_{v,w} = \langle M_v^t\cdot \alpha_w,M_w^t\cdot \alpha_v \rangle \subset V_{n+m}. 
\end{equation}


The following lemma aims to explain the geometric role of $\pi_{v,w}$ just defined.

\begin{lemma}\label{lemma.Hilbertwellposed}
Let $[v],[w]$ be two distinct points in $S$. Then:
\begin{enumerate}[{\normalfont (1)}]
    \item $\dim\pi_{v,w}=0$ if and only if the line $\ell_{v,w}$ joining $[v],[w]$ is entirely contained in $S$, and $\psi(\ell_{v,w})$ reduces to a point in $\PP^{n-1}$.
    \item $\dim\pi_{v,w}=1$ if and only if the line $\ell_{v,w}$ joining $[v],[w]$ is entirely contained in $S$, and $\psi(\ell_{v,w})$ is a line in $\PP^{n-1}$.
    \item $\dim\pi_{v,w}=2$ if and and only if $\psi(\ell_{v,w})\subset \PP^{n-1}$ intersects the line between $[\alpha_v]$ and $[\alpha_w]$ in precisely two points.
\end{enumerate}
\end{lemma}

\begin{proof}
We proceed case by case.
\begin{enumerate}
    \item $\langle\alpha_v\rangle=\langle\alpha_w\rangle$ if and only if $M_v^t\cdot\alpha_w=M_w^t\cdot\alpha_v=0$; therefore $\dim\pi_{v,w}=0$ if and only if $\dim\langle\alpha_v,\alpha_w\rangle=1$ and the statement follows by
    \[ 
    M^t_{\lambda v+\mu w}\cdot \alpha_v = M^t_{\lambda v}\cdot \alpha_v + M^t_{\mu w}\cdot \alpha_v = 0 
    \]
    for every $\lambda,\mu\in\mathbb{C}$.
    \item If $\dim\pi_{v,w}=1$ then there exist $\delta_1,\delta_2\in\mathbb{C}$ such that $\delta_1M_w^t\cdot\alpha_v+\delta_2M_v^t\cdot\alpha_w=0$. Therefore $M_{\lambda v+\mu w}^t(\lambda\delta_1\alpha_v+\mu\delta_2\alpha_w)=\lambda\mu\left(\delta_2M^t_{v}\cdot \alpha_w+\delta_1M^t_{w}\cdot\alpha_v\right)=0$, so that $[\lambda v+\mu w]\in S$ for every $\lambda,\mu\in\mathbb{C}$ and the kernels of the transpose matrices are aligned in $\PP^{n-1}$.\\
    For the converse, first notice that $\dim\pi_{v,w}\neq 0$. Moreover, if there exists another point $[u]\in \ell_{v,w}\cap S$ with $\alpha_u=\delta_1\alpha_v+\delta_2\alpha_w$, and $u=\lambda v+\mu w$. Then $\lambda\delta_2 M^t_v\cdot\alpha_w+\mu\delta_1 M^t_w\cdot\alpha_v=0$, so that $\dim\pi_{v,w}=1$.
    \item By contradiction, suppose there exists a third point $[u]\in \ell_{v,w}\cap S$ with $\alpha_u=\delta_1\alpha_v+\delta_2\alpha_w$, and $u=\lambda v+\mu w$. Then $\lambda\delta_2 M^t_v\cdot\alpha_w+\mu\delta_1 M^t_w\cdot\alpha_v=M^t_u\cdot\alpha_u=0$, so that $\dim\pi_{v,w}\leq 1$. Viceversa, if $\dim\pi_{v,w}\leq 1$ then $\ell_{v,w}\subset S$ by the above items so that $\psi(\ell_{v,w})$ intersects the line between $[\alpha_v],[\alpha_w]$ either in one point or in infinite points.\qedhere
\end{enumerate}
\end{proof}

\begin{definition}\label{def.omegazerolocus}
We shall use the shorthand notation
\[ 
G_{n,s,m}=\Gr(2,n)\times \Gr(2,s+1)\times \Gr(n+m-2,n+m),
\]
and we shall denote with the same letter $\mU$ the tautological (sub)bundle on each Grassmannian.  
There is a natural section $\omega \in \mathrm{H}^0(G_{n,s,m},\mU^{\vee} \boxtimes \mU^{\vee} \boxtimes \mU^{\vee})$ associated to $M$, defined by
\[ 
\omega\colon V_n\otimes V_{s+1}\otimes V_{n+m} \longrightarrow \mathbb{C},
\qquad (a,u,b) \mapsto a^t\cdot M_u\cdot b. 
\]
We denote by $Z=V(\omega)\subset G_{n,s,m}$ its zero scheme.
\end{definition}

We note that there is an identity
\[
Z= \left\{ P\in G_{n,s,m} \,\colon \, \omega\big|_{P} \equiv 0 \right\}
\]
where, if $P=(\rho_1,\rho_2,\rho_3)$, then $\omega|_P\equiv 0$ means that $\omega(a,u,b)=0$ for every $a\in\rho_1$, $u\in\rho_2$, $b\in\rho_3$.

\begin{definition}\label{def.alternativeHilbertmap}
To any pair of distinct points $[v],[w]\in S$ such that $\dim\pi_{v,w}=2$ we can associate the point
\[ 
P_{[v],[w]} = \left( \langle\alpha_v \, , \, \alpha_w\rangle \; , \; \langle v,w\rangle \; , \; \pi_{v,w}^{\perp} \; \right) \in G_{n,s,m},
\]
where $\pi_{v,w}$ is as defined in \Cref{eqn:crucial_plane}.
\end{definition}

\begin{remark}
By \Cref{lemma.Hilbertwellposed}, there is an immersion
\[ 
H\to G_{n,s,m}, \quad [v]+[w] \mapsto P_{[v],[w]},
\]
where $H=\{ [v]+[w] \in \Sym^2(S) \setminus S \,|\, \dim\pi_{v,w}=2 \}\subset \Sym^2(S)$.
\end{remark}

\begin{lemma}
Let $P_{[v],[w]}$ be as in \Cref{def.alternativeHilbertmap}, then
$P_{[v],[w]}\in Z$.
\end{lemma}

\begin{proof}
We need to prove that $\omega|_{P_{[v],[w]}} \equiv 0$.
Let $a=h_1\alpha_v+h_2\alpha_w$ and $u=\lambda v+\mu w$. Then
\begin{align*}
    \omega(a,u,b) 
    &= (h_1\alpha^t_v+h_2\alpha^t_w)\cdot M_{\lambda v+\mu w} \cdot b \\
    &= h_1\mu \,(\alpha_v^t\cdot M_w) \cdot b + h_2\lambda\, (\alpha_w^t\cdot M_v)\cdot b \\
    &= 0.\qedhere
\end{align*}
\end{proof}

\subsection{Main technical result}

In the following, given $\rho\in \Gr(2,k)$ we shall denote by $[\rho]\subset\mathbb{P}^{k-1}$ the projective line defined by $\rho$. Also, given two distinct points $[v]$ and $[w]$ in $\PP^s$, we shall denote by $\ell_{v,w} \subset \PP^s$ the line connecting them.

The following is the main technical result of the paper.

\begin{theorem}\label{thm.omegazerolocus}
Let $P=(\rho_1\,,\rho_2\,,\rho_3)\in Z$.
Then one of the following holds:
\begin{itemize}
\item [\textbf{a}.] there exist two (and only two) distinct points $[v],[w]\in S\cap [\rho_2]$ such that $P=P_{[v],[w]}$,
\item [\textbf{b}.] $[\rho_2]\subset S$ and $\psi([\rho_2])$ reduces to a point in $[\rho_1] \subset \PP^{n-1}$,
\item [\textbf{c}.] there exists exactly one point $[v]\in S$ where $[\rho_2]$ is tangent and such that $[\alpha_v]\in[\rho_1]$.
\item [\textbf{d}.] $[\rho_2]\subset S$ and $[\rho_1]=\psi(\ell_{v,w})\subset \PP^{n-1}$.
\end{itemize}
\begin{proof}
Consider the linear subspace
\[ 
W_{(\rho_1,\rho_2)} = \Set{ M^t_u\cdot a \, \vert\, a\in\rho_1, \, u\in \rho_2} \subset V_{n+m}.
\]
Now, since $(\rho_1,\rho_2,\rho_3)\in Z$, we have $W_{(\rho_1,\rho_2)} \subset \rho_3^{\perp}$ so that $\dim W_{(\rho_1,\rho_2)}\leq 2$. Let us proceed case by case.

Suppose $\dim W_{(\rho_1,\rho_2)}=0$ first. This means that $M_u^t\cdot a=0$ for every $u\in\rho_2$ and for every $a\in\rho_1$. This is impossible since it would imply $\dim\ker(M_u^t)\geq 2$, i.e.~$\rank (M_u)<n-1$. But this is in contradiction with the generality assumption on $M$.

Next, let us suppose $\dim W_{(\rho_1,\rho_2)}=1$. This means that we can find a basis $\{M_{u_1}^t\cdot {a_1}\}$ for $W_{(\rho_1,\rho_2)}$. We can complete to bases $\{a_1,a_2\}\subset \rho_1$ and $\{u_1,u_2\}\subset \rho_2$, in such a way that 
    \[
    M_{u_1}^t\cdot a_2=M_{u_2}^t\cdot a_1=0 \;. 
    \]
    In fact, if $\{a_1,a_2'\}$ is any basis for $\rho_1$, then
    $M_{u_1}^t\cdot (h a_1+a_2') = 0$
    for some $h\in \mathbb{C}$. Hence it is sufficient to chose $a_2=ha_1+a_2'$. A similar argument provides the required choice of $u_2\in\rho_2$. In particular, $[u_1],[u_2]\in S$ and by assumption $M_{u_1}^t\cdot a_1+M_{u_2}^t\cdot a_2=0$ (up to a possible rescale of $a_2$). Therefore
    \[ M_{\lambda u_1+\mu u_2}^t\cdot (\mu a_1+\lambda a_2)= \lambda\mu (M_{u_1}^t\cdot a_1+M_{u_2}^t\cdot a_2) = 0 \;, \]
    so that $[\rho_2]\subset S$ and $[\rho_1]=\psi(\ell_{v,w})\subset \PP^{n-1}$. This is the case $\boldit{d}$ in the statement.

Finally, let us suppose $\dim W_{(\rho_1,\rho_2)}=2$, which means $W_{(\rho_1,\rho_2)}=\rho_3^{\perp}$. Notice that we can choose bases $\{a_1,a_2\}\subset \rho_1$ and $\{u_1,u_2\}\subset \rho_2$, in such a way that
    \[ 
    W_{(\rho_1,\rho_2)} = \langle\nu_{1,1},\nu_{2,1}\rangle = \langle\nu_{1,2},\nu_{2,2}\rangle,
    \]
    where
    \[ 
    \nu_{i,j}=M_{u_i}^t\cdot a_j, \qquad i,j\in\{1,2\} \; .
    \]
    In fact, if $\dim\langle \nu_{1,1}\, ,\, \nu_{2,1}\rangle = 1$ then there exist $\delta_1,\delta_2\in \mathbb{C}$ such that $\delta_1\nu_{1,1}+\delta_2\nu_{2,1}=0$. It follows that$M_{\delta_1 u_1+\delta_2 u_2}^t\cdot a_1=0$ so that $[\delta_1 u_1+\delta_2 u_2]\in S$. Now if $\delta_1\neq 0$ we define $u_1'=\delta_1u_1+\delta_2u_2$ and we replace the basis $\{u_1,u_2\}$ with $\{u_1',u_2\}$. Similarly, assuming $\dim\langle\nu_{1,2}\, ,\, \nu_{2,2}\rangle = 1$ one can eventually replace $\{a_1,a_2\}$ with $\{a_1,a_2'\}$.

Hence there exists a matrix $\Phi \in \mathsf{Mat}_{2,2}(\mathbb C)$ realising a coordinate change
\[ \begin{pmatrix}
\nu_{1,2}&\nu_{2,2}
\end{pmatrix}
=-\begin{pmatrix}
\nu_{1,1}&\nu_{2,1}
\end{pmatrix} \cdot
\Phi,
\]
where we adopted the notation $\begin{pmatrix}\nu_{1,j}& \nu_{2,j}\end{pmatrix}$ to denote the $(n+m)\times 2$ matrix whose columns are $\nu_{1,j}$ and $\nu_{2,j}$. Our aim is now to study vectors $v\in\rho_2$ corresponding to points in $[v]\in S$ with the additional property that $[\alpha_v]\in[\rho_1]$. Such a point is given by the choice of a nonzero vector 
\[
\begin{pmatrix}\lambda\\ \mu\end{pmatrix}\in\mathbb{C}^2
\]
together with scalars $\delta_1,\delta_2\in\mathbb{C}$, not both vanishing, such that
\[ 
M^t_{\lambda u_1+\mu u_2}\cdot (\delta_1 a_1+\delta_2 a_2) = 0.
\]
In particular, it is not restrictive to assume $\delta_2\neq 0$ since $\dim\langle\nu_{1,1}\, ,\,\nu_{2,1}\rangle = 2$. Rename $\delta=\delta_1\delta_2^{-1}\in\mathbb{C}$ and consider the following equalities:
\[ 
\begin{aligned}
M_{\lambda u_1+\mu u_2}^t\cdot (\delta a_1 + a_2) &= \lambda\delta M_{u_1}^t\cdot a_1 + \mu\delta M_{u_2}^t\cdot a_1 + \lambda M_{u_1}^t\cdot a_2 + \mu M_{u_2}^t\cdot a_2  \\
&= \lambda\delta\nu_{1,1}+\mu\delta\nu_{2,1}+\lambda\nu_{1,2}+\mu\nu_{2,2}\\
&= \begin{pmatrix}\nu_{1,1}&\nu_{2,1}\end{pmatrix}\cdot \begin{pmatrix}\delta\lambda\\ \delta\mu\end{pmatrix} + \begin{pmatrix} \nu_{1,2}&\nu_{2,2}\end{pmatrix}\cdot\begin{pmatrix}\lambda\\ \mu\end{pmatrix}  \\
&= \begin{pmatrix}\nu_{1,1}&\nu_{2,1}\end{pmatrix}\cdot \begin{pmatrix}\delta\lambda\\ \delta\mu\end{pmatrix}-
\begin{pmatrix} \nu_{1,1}&\nu_{2,1}\end{pmatrix}\cdot
\Phi\cdot
\begin{pmatrix}\lambda\\ \mu\end{pmatrix}  \\
&= \begin{pmatrix}\nu_{1,1}&\nu_{2,1}\end{pmatrix}\cdot\left\{ \delta \cdot \mathrm{id} - \Phi
\right\}\cdot\begin{pmatrix}\lambda\\ \mu\end{pmatrix}.
\end{aligned}
\]
Now, since $\dim\langle \nu_{1,1},\nu_{2,1}\rangle=2$ the last line vanishes if and only if $\delta$ is an eigenvalue of $\Phi$ and $\begin{pmatrix}\lambda& \mu\end{pmatrix}^t$ is an eigenvector relative to $\delta$.
Since $\mathbb{C}$ is algebraically closed, we conclude that the line $[\rho_2]\subset\mathbb{P}^s$ always intersects $S$ in (at least) one point $[v]=[\lambda u_1+\mu u_2]$ satisfying $[\alpha_v]\in[\rho_1]$.
More precisely we have the following three possibilities.

\smallbreak
\begin{itemize}
\item [$\boldit{a}$.] \emph{The matrix $\Phi$ admits two different eigenvalues $\delta$ and $\theta$.}\\
In this case we have two (independent) eigenvectors 
\[
\begin{pmatrix} \lambda_{\delta}\\ \mu_{\delta}\end{pmatrix}, \quad \begin{pmatrix} \lambda_{\theta}\\ \mu_{\theta}\end{pmatrix}
\]
and the above discussion provides two distinct points
\begin{align*}
[v]&=[\lambda_{\delta}u_1+\mu_{\delta}u_2]\in S\cap[\rho_2], \\ [w]&=[\lambda_{\theta}u_1+\mu_{\theta}u_2]\in S\cap[\rho_2].
\end{align*}
Notice that by \Cref{lemma.Hilbertwellposed} either we are in case $\boldit{d}$ of the statement or the points $[v],[w]\in S\cap[\rho_2]$ are the only ones satisfying the additional property $[\alpha_v],[\alpha_w]\in[\rho_1]$.
Clearly, in this last case $\rho_1=\langle \alpha_v,\alpha_w\rangle$, $\rho_2=\langle v,w\rangle$, and $W_{(\rho_1,\rho_2)} = \langle M_v^t\alpha_w, M_w^t\alpha_v\rangle=\pi_{v,w}$; therefore $(\rho_1,\rho_2,\rho_3)=P_{[v],[w]}$. This is item $\boldit{a}$ in the statement.
\item [$\boldit{b}$.] \emph{The matrix $\Phi$ admits one eigenvalue $\delta$ whose eigenspace is $2$-dimensional.}\\
In this case every non-trivial $\begin{pmatrix}\lambda& \mu\end{pmatrix}^t\in \mathbb{C}^2$ is an eigenvector so that the line defined by $[\rho_2]$ in $\mathbb{P}^s$ is entirely contained in $S$. On the other hand the matrix $M_v^t$ admits the same kernel $\delta a_1+a_2\in \rho_1$ for every $v\in\rho_2$. This is item $\boldit{b}$ in the statement.
\item [$\boldit{c}$.] \emph{The matrix $\Phi$ admits only one eigenvalue $\delta$ whose eigenspace is $1$-dimensional.}\\
In this case any eigenvector $\begin{pmatrix}\lambda& \mu\end{pmatrix}^t$ corresponds to the same point
$[v]=[\lambda u_1+\mu u_2]\in S$.
Hence $[v]$ is the only point in the intersection $[\rho_2]\cap S$ such that $[\alpha_v]\in[\rho_1]$. Moreover, in this case the \emph{algebraic} multiplicity of $\delta$ is $2$; i.e.~the multiplicity of the intersection $[\rho_2]\cap S$ is $2$ at $[v]$.
This is item $\boldit{c}$ in the statement.
\end{itemize}
The proof is now complete.
\end{proof}
\end{theorem}

\section{Existence of special lines}\label{section.badlines}
As in the previous section, we fix integers $n\geq 3$, $m\geq 0$, $s\in\{m+2\, ,\dots,\, 2m+3\}$ and a general map of vector bundles $\varphi$ as in \eqref{eqn:bundlemap}. Moreover, we shall use the following terminology: a line $\ell \subset S = D_{n-1}(\varphi)\subset \mathbb P^s$ is said to be \emph{of type} $\boldit{b}$ (resp.~\emph{of type} $\boldit{d}$) if it arises from a point $P \in Z$ satisfying condition $\boldit{b}$ (resp.~condition $\boldit{d}$) in \Cref{thm.omegazerolocus}.

\subsection{Excluding lines of type \texorpdfstring{$\boldit{b}$}{}}
The first aim of this section is to understand the fibres of the map $\psi$, and we will be particularly interested in the existence of points $[\alpha]\in\mathbb{P}^{n-1}$ whose fibre $\psi^{-1}([\alpha])$ is a line in $S$.

Fix $[\alpha]\in\mathbb{P}^{n-1}$ and observe that
\begin{equation}\label{eqn:fibre_psi}
\psi^{-1}([\alpha])=\Set{[v]\in S\,\vert\, M^t_v\cdot\alpha=0} \subset S 
\end{equation}
is nothing but the solution set of a linear system of $n+m$ equations in $s+1$ variables, namely an intersection of $n+m$ hyperplanes in $\mathbb{P}^s$. Therefore the fibre \eqref{eqn:fibre_psi} is always linear. Moreover, it can be described by means of a matrix $A_{\alpha}\in \mathsf{Mat}_{n+m,s+1}(\mathbb C)$, and by the linearity with respect to $\alpha$ we get an immersion
\begin{equation}\label{eqn:map_f}
f\colon \mathbb{P}^{n-1} \hookrightarrow \mathbb{P}=\mathbb{P}\left(\mathsf{Mat}_{n+m,s+1}(\mathbb C)\right),\quad [\alpha]\mapsto [A_{\alpha}]. 
\end{equation}

\begin{remark}
Notice that an additional condition $s\leq n+m$ is essential in order to obtain $0$-dimensional fibres of $\psi$, and similarly $s\leq n+m+1$ is necessary in order to obtain $1$-dimensional fibres, as well as $s\leq n+m+2$ for $2$-dimensional fibres.
\end{remark}

Let us denote by $N_k\subset \mathbb{P}$  the subvariety of matrices of rank at most $k$. We can easily compute the codimension of $N_k$ in $\mathbb{P}$ as
\[ 
\codim(N_k)= (n+m-k)(s+1-k), 
\]
so that in particular assuming $s\leq n+m$ one finds
\begin{align*}
    \codim(N_s)&=n+m-s\geq 0\\
    \codim(N_{s-1})&= 2(n+m-s+1)\geq 2\\
    \codim(N_{s-2})&= 3(n+m-s+2)\geq 6.
\end{align*}  

\begin{theorem}\label{thm.badlines}
Let $\psi\colon S\to \mathbb{P}^{n-1}$ and $f\colon\mathbb{P}^{n-1}\hookrightarrow\mathbb{P}$ be the maps defined by \Cref{lemma.psi} and \eqref{eqn:map_f} respectively. Fix integers $n\geq 3$, $m\geq 0$, $s\in\{m+2\, ,\dots,\, 2m+3\}$. 
\begin{itemize}[{\normalfont (i)}]
\item [\emph{(i)}] Assume $s=n+m$. Then
    \begin{enumerate}[{\normalfont (1)}]
    \item $\psi$ is surjective and its generic fibre is a point.
    \item $f\circ\psi$ admits $1$-dimensional fibres precisely over $\mathrm{Im}(f)\cap N_{s-1}$.
    \end{enumerate}
\item [\emph{(ii)}] Assume $s<n+m$. Then
    \begin{enumerate}[{\normalfont (1)}]
    \item $\psi$ is a closed immersion if and only if $n>2s-2m-3$, in which case the image of the composition $f\circ\psi$ is $\mathrm{Im}(f)\cap N_s\subset \mathbb{P}$,
    \item $f\circ\psi$ admits $1$-dimensional fibres if and only if $n\leq 2s-2m-3$, and such fibres arise precisely over $\mathrm{Im}(f)\cap N_{s-1}$,
    \item $f\circ\psi$ admits $2$-dimensional fibres if and only if $n\leq \frac{1}{2}(3s-3m-7)$, and such fibres arise precisely over $\mathrm{Im}(f)\cap N_{s-2}$,
    \end{enumerate}
\end{itemize}
\begin{proof}
Let us proceed by steps.
\begin{itemize}
    \item [(i)] First suppose that $s=n+m$.
    As already observed the fibre $\psi^{-1}([\alpha])$ is cut by $n+m$ hyperplanes in $\PP^s$, hence the generic fibre reduces to a point. Moreover, the fibre is $1$-dimensional at those $[\alpha]$ such that $f([\alpha])\in\mathrm{Im}(f)\cap N_{s-1}\subset \mathbb{P}$, which has dimension $(n-1)-2=n-3\geq 0$.
    \item [(ii)] Now assume $s<n+m$. Then the fibre $\psi^{-1}([\alpha])$ is a point (respectively a line) precisely at those $[\alpha]$ such that $f([\alpha])\in\mathrm{Im}(f)\cap N_k\subset \mathbb{P}$ with $k=s< n+m$ (respectively $k=s-1< n+m$).
    Therefore the image of $\psi$ describes a subvariety of $\mathbb{P}^{n-1}$ of dimension
    \[ \dim\psi(S) = (n-1)-\codim(N_s)=(n-1)-(n+m-s)=s-m-1 = \dim(S) \; , \]
    while the $1$-dimensional fibres of $\psi$ (if they exist) are mapped onto a locus of dimension
    \[ (n-1)-\codim(N_{s-1})=(n-1)-2(n+m-s+1) = 2s-n-2m-3 \; . \]
    The condition $n>2s-2m-3$ is the same as $\codim(N_{s-1})=2(n+m-s+1)>n-1$, which in turn is equivalent to require that $N_{s-1}$ is empty; here we are using the genericity of the original matrix $M$ (hence of the form $\omega$) from which it follows the genericity of the immersion of $\mathbb{P}^{n-1}$ in $\mathbb{P}$ through $f$. Hence the fibres of the map consist of at most one point if and only if $n>2s-2m-3$, in which case $\psi$ is a closed immersion, as wanted.
    
    Finally, the fibres of dimension at least $2$ arise over $\mathrm{Im}(f)\cap N_{s-2}$, for which the expected dimension is
    \[
    (n-1)-\codim(N_{s-2}) = (n-1)-3(n+m-s+2)=3s-2n-3m-7. 
    \]
    This number is non-negative if and only if $n\leq \frac{1}{2}(3s-3m-7)$, as required.\qedhere
\end{itemize}
\end{proof}
\end{theorem}

\subsection{Excluding lines of type \texorpdfstring{$\boldit{d}$}{}}
The next aim of this section is to show that the lines described by item $\boldit{d}$ of \Cref{thm.omegazerolocus} do not occur whenever $n > 2s-3m-2$. Recall that these are the lines $\ell\subset S$ such that the image $\ell'=\psi(\ell)$ remains a line in $\PP^{n-1}$.

\begin{theorem}\label{thm.verybadlines}
Let $n\geq3$, $m\geq0$ and $s\in\{m+2\, ,\dots,\, 2m+3\}$.
\begin{itemize}
    \item If $n>2s-3m-1$ then the composition
\[ 
\begin{tikzcd}
    \pi_Z\colon Z \arrow[hook]{r}{\iota} & 
    G_{n,s,m} \arrow{r}{\mathrm{pr_{12}}} &
    \Gr(2,n)\times\Gr(2,s+1)
\end{tikzcd}    
\]
is injective, where $\iota$ is the natural inclusion and $\mathrm{pr_{12}}$ is the natural projection.
    \item A line $\ell\subset S\subset \PP^s$ such that $\ell'=\psi(\ell)$ remains a line in $\PP^{n-1}$ exists if and only if the map $\pi_Z$ admits an $(n+m-2)$-dimensional linear fibre.
\end{itemize}
\begin{proof}
Let us proceed by steps.
\begin{itemize}
    \item The fibre of $\pi_Z$ over $(\rho_1,\rho_2)\in \Gr(2,n)\times\Gr(2,s+1)$ can be easily described as
    \[ 
    \pi_Z^{-1}(\rho_1,\rho_2)=\left\{(\rho_1,\rho_2,\rho_3)\in G_{n,s,m} \,\Bigl\vert\, \rho_3\subset W_{(\rho_1,\rho_2)}^{\perp} \right\} 
    \]
    where $W_{(\rho_1,\rho_2)} = \langle M_u^t\cdot a\,|\,a\in\rho_1,\, u\in\rho_2\rangle\subset V_{n+m}$. Notice that
    \[
    W_{(\rho_1,\rho_2)} = \langle M_{u_i}^t\cdot a_j\,|\,1\leq i,j\leq 2\rangle
    \]
    where $\{a_1,a_2\}$ and $\{u_1,u_2\}$ are arbitrary bases for $\rho_1$ and $\rho_2$ respectively. In particular, $\dim W_{(\rho_1,\rho_2)}\leq 4$. Since $\rho_3\in \Gr(n+m-2,n+m)$, we deduce
    \[ \pi_Z^{-1}(\rho_1,\rho_2)=\begin{cases}
    \emptyset & \mbox{ if } \dim W_{(\rho_1,\rho_2)} \geq 3 \\
    \bigl(\rho_1,\rho_2,W_{(\rho_1,\rho_2)}^{\perp}\bigr) & \mbox{ if } \dim W_{(\rho_1,\rho_2)} =2\\
    \Gr\bigl(n+m-2,W_{(\rho_1,\rho_2)}^{\perp} \bigr) & \mbox{ if } \dim W_{(\rho_1,\rho_2)}=1
    \end{cases} \]
    Notice that $\dim W_{(\rho_1,\rho_2)}\neq 0$, otherwise we would have points $u\in V_{s+1}$ satisfying $\rank M_u=n-2$ and this is excluded since $D_{n-2}(\varphi)=\emptyset$.
    In particular, for $\dim W_{(\rho_1,\rho_2)}=1$ we have
    \[ 
    \Gr\left(n+m-2,W_{(\rho_1,\rho_2)}^{\perp} \right)\cong\PP^{n+m-2}.
    \]
    On the other hand $Z$ cannot contain an $(n+m-2)$-dimensional subspace whenever $\dim(Z)=2(s-m-1) < n+m-2$, i.e.~when $n>2s-3m$.
    Moreover, in the case $n=2s-3m$, the irreducibility of $Z$ together with the non injectivity of the map $\pi_Z$ would imply $Z\cong \PP^{n+m-2}$, which is impossible because otherwise the map $\pi_Z$ would be constant so that $S$ would reduce to a line $S=[\rho_2]\cong \PP^1\subset \PP^s$. Of course this is false being $n\geq3$.
    \item We claim that the existence of a line $\ell\subset S$ such that $\ell'=\psi(\ell)$ remains a line in $\PP^{n-1}$ is equivalent to the existence of a $(n+m-2)$-dimensional fibre of the map $\pi_Z$. In fact, by \Cref{lemma.Hilbertwellposed} the existence of such a line $\ell$ 
    is equivalent to a point 
    \[ (\rho_1,\rho_2)\in \Gr(2,n)\times\Gr(2,s+1) \]
    with $[\rho_1]=\ell'$ and $[\rho_2]=\ell$, that moreover satisfies $\dim W_{(\rho_1,\rho_2)}=1$.
    As shown in the first item this is equivalent to the condition $\dim \pi_Z^{-1}(\rho_1,\rho_2) = n+m-2$. \qedhere
\end{itemize}
\end{proof}
\end{theorem}

In \Cref{cor.omegazerolocus} we will be able to give a better bound than the one in \Cref{thm.verybadlines} in the cases $m=0$ and $m=1$.

\begin{remark}\label{rmk.boundcompare}
Notice that for large values of $m$, the bound $n>2s-2m-3$ obtained in \Cref{thm.badlines} is stronger than the one obtained in \Cref{thm.verybadlines}. More precisely,
\[ n>2s-2m-3 \qquad \Longrightarrow \qquad n>2s-3m-1 \]
as soon as $m\geq2$.
\end{remark}

\subsection{Conclusions}
We now summarise the main results of this section in the following corollary.

\begin{corollary}\label{cor.omegazerolocus}
Let $n\geq3$, $m\geq0$ and $s\in\{m+2\,,\dots,\,2m+3\}$. Assume $n>2s-2m-3$ and let $(\rho_1\,,\rho_2\,,\rho_3)\in Z$.\\
Then one of the following holds:
\begin{enumerate}[{\normalfont (1)}]
\item there exist two (and only two) distinct points $[v],[w]\in S\cap [\rho_2]$ such that $P=P_{[v],[w]}$,
\item there exists exactly one point $[v]\in S$ where $[\rho_2]$ is tangent and such that $[\alpha_v]\in[\rho_1]$.
\end{enumerate}
\begin{proof}
If $m\geq 2$, then by \Cref{rmk.boundcompare} the statement is an immediate consequence of \Cref{thm.omegazerolocus}, \Cref{thm.badlines}, \Cref{thm.verybadlines}.

For $m=1$ our assumption becomes $n>2s-5$, so that the hypothesis of \Cref{thm.badlines} are satisfied while \Cref{thm.verybadlines} works as soon as $n>2s-4$. Let us prove by hand that choosing $m=1$ and $n=2s-4$ the map
\[ 
\begin{tikzcd}
    \pi_Z\colon Z \arrow[hook]{r}{\iota} & 
    G_{n,s,m} \arrow{r}{\mathrm{pr_{12}}} &
    \Gr(2,n)\times\Gr(2,s+1)
\end{tikzcd}    
\]
is still injective. Here $\iota$ is the natural inclusion and $\pi$ is the natural projection. The idea is to exclude high dimensional fibres following the proof of \Cref{thm.verybadlines}.
\begin{itemize}
\item[(A)] Set $s=5, n=6, m=1$. We have to exclude the existence of a $\PP^5 \subset Z$. However, in this case $Z$ is a 6-fold with $h^{2,0}=h^{0,2}=0$ and $h^{1,1}=3$, which in this case is equal to the Picard rank. In fact, the $\mathrm{Pic}(\Z)$ is generated by the restrictions of the three Pl\"ucker line bundles from the ambient Grassmannians. Hence, by degree reasons, since $Z$ is smooth, it cannot contain a $\PP^5$.
\item[(B)] Set $s=4, n=4, m=1$. We have to exclude the existence of a $\PP^3 \subset Z$. This time, we know that $\mathrm{Pic}(S)$ is only generically of rank 2, and the same holds for $Z$ (in fact $h^{2,0}(Z)=4$). For the same reasons above, we can therefore exclude the existence of a $\PP^3$ for a general $Z$. But this is enough, since we started by hypothesis from a general matrix, and $S$ - which is a isomorphic to a determinantal quintic hypersurface also described as a complete intersection in $\PP^4 \times \PP^3$ - in this case the Picard group will be $\Z^2$ and generated by the two classes hyperplane classes of $\PP^4$ and $\PP^3$, and won't even contain lines by a Noether-Lefschetz type argument, see \cite{lopez1991noether} and also \cite[Theorem 1.2]{ciliberto2019lines}. Similarly $Z$ won't contain a copy of $\PP^3$.
\end{itemize}
Hence the statement is proven for every $m\geq1$. We are only left with the case $m=0$. In this case our assumption becomes $n>2s-3$, so that the hypothesis of \Cref{thm.badlines} are satisfied while \Cref{thm.verybadlines} works as soon as $n>2s-1$. Hence we only need to check the following cases, where $m=0$ and either $n=2s-1$ or $n=2s-2$.
\begin{itemize}
    \item[(C)] Set $s=2,n=3,m=0$. Just observe that in this case $S$ is a curve of genus $g\neq0$, so that in particular it does not contain lines and the conclusion follows by the second item in \Cref{thm.verybadlines}.
    \item[(D)] Set $s=3,n=5,m=0$. We have to exclude the existence of $\PP^3\subset Z$. To this aim, it is sufficient to run exactly the same argument as in case (B).
    \item[(E)] Set $s=3,n=4,m=0$. In this case $S$ and $\psi(S)$ in $\PP^3$ are precisely the K3 surfaces studied by Oguiso in~\cite{oguiso2012, oguiso2017isomorphic}.
    Notice that a generic determinantal K3 surface does not contain lines, since the Picard lattice is
    \[
    \begin{pmatrix} 4&6\\6&4\end{pmatrix}
    \]
    and the square of every other element is divisible by $4$. Therefore we do not have $(-2)$-curves in general.
    Hence the second item in \Cref{thm.verybadlines} implies the injectivity of the map $\pi_Z\colon Z\to \Gr(2,4)\times \Gr(2,4)$ as required.\qedhere
\end{itemize}
\end{proof}
\end{corollary}

\section{Hilbert squares of degeneracy loci}\label{section.mapZHilb}
In this section we finally prove our main theorem, namely \Cref{thm:mainthm_intro}.

We denote by $\Fl(1,2,n)$ and by $\Fl(1,2,s+1)$ the appropriate flag varieties. Moreover, we denote by 
\[
\Gamma_{\psi} \subset \mathbb P^{n-1} \times S\subset \PP^{n-1}\times \PP^{s}
\]
the graph of the morphism $\psi$ of \Cref{lemma.psi}. 

In the category of $\mathbb{C}$-schemes, we consider the limit $\mathcal{V}$ of the following diagram of solid arrows
\[
\begin{tikzcd}[column sep=tiny]
& & \mathcal V\arrow[dotted]{d}\arrow[dotted]{dll}\arrow[dotted]{drr} & & \\
\Gamma_\psi\arrow{dr} & & \Fl(1,2,n)\times \Fl(1,2,s+1)\arrow{dl}\arrow{dr} & & Z\arrow{dl} \\
&  \PP^{n-1}\times\PP^s & & \Gr(2,n)\times\Gr(2,s+1) & 
\end{tikzcd}
\]
Notice that, set-theoretically, $\mathcal{V}$ can be described as
\[ 
\mathcal{V} 
= \Set{ \left([v],\psi([v]),(\rho_1,\rho_2,\rho_3)\right) \in \Gamma_{\psi}\times Z \,|\, [v]\in[\rho_2], \psi([v])\in[\rho_1]} \hookrightarrow \Gamma_{\psi}\times Z
\]
and via the natural isomorphism $\Gamma_\psi \simto S$ we make the identification
\[
\mathcal{V} = \Set{ \left([v],(\rho_1,\rho_2,\rho_3)\right) \,|\, [v]\in[\rho_2], \psi([v])\in[\rho_1]} \hookrightarrow S \times Z.
\]
Composing with the projection $S\times Z \to Z$, we obtain a morphism
\[ 
\pi\colon \mathcal{V} \hookrightarrow S\times Z \to Z.
\]
We now show that this morphism defines a modular map $Z \to \Hilb^2(S)$.

\begin{lemma}\label{lemma.lenght2}
Let $n\geq3$, $m\geq0$, $s \in \set{m+2\,,\dots,\,2m+3}$. Assume $n > 2s-2m-3$. Then the natural morphism 
\[
\pi \colon \mathcal V \hookrightarrow S \times Z \to Z
\]
is a flat family of length $2$ subschemes of $S$.
\begin{proof}
Since $Z$ is smooth, in particular reduced, it is enough to prove that the fibre over any closed point is a finite subscheme of length $2$. Flatness is then automatic. 

In fact, the fibre $\pi^{-1}(\rho_1,\rho_2,\rho_3)$ over a point $P=(\rho_1,\rho_2,\rho_3) \in Z$ is of the form
\[
\pi^{-1}(P) =  \set{([v],P)|[v]\in S\cap [\rho_2],\,[\alpha_v] \in [\rho_1]}.
\]
By \Cref{cor.omegazerolocus}, this is a length 2 subscheme of $S \times \set{P}$ if $n>2s-2m-3$, which we are assuming.
\end{proof}
\end{lemma}

In particular, if $n > 2s-2m-3$, the morphism $\pi$ gives rise, via the universal property of the Hilbert scheme, to a morphism
\[
\vartheta \colon Z \to \Hilb^2(S).
\]


\begin{theorem}\label{thm.Hilbiso}
Let $n\geq3$, $m\geq0$, $s \in \set{m+2\,,\dots,\,2m+3}$. Assume $n > 2s-2m-3$. Then the morphism $\vartheta \colon Z \to \Hilb^2(S)$ is an isomorphism.
\begin{proof}
To prove $\vartheta$ is an isomorphism, by Zariski's Main Theorem it is enough to prove it is bijective, since both source and target are smooth $\mathbb C$-varieties of the same dimension $2(s-m-1)$.

On $\mathbb C$-valued points, the morphism $\vartheta$ is defined by
\[
\vartheta(P) = [\pi^{-1}(P)] \,\in\,\Hilb^2(S).
\]
By the uniqueness conditions spelled out in \Cref{cor.omegazerolocus}, the map $\vartheta $ is injective. By the same argument, one can see that $\vartheta (B)$ is an injective map of sets for every $\mathbb C$-scheme $B$. Thus $\vartheta$ is a proper monomorphism, i.e.~a closed immersion. 

Since source and target are smooth of the same dimension, $\vartheta$ is an lci morphism of codimension $0$, hence the tangent map $T_Z \to \vartheta^\ast T_{\Hilb^2(S)}$ is an isomorphism, in particular $\vartheta$ is \'etale. Thus it is an open and closed map to a connected scheme, hence it is surjective.
\end{proof}
\end{theorem}

\section{Geometric examples}\label{section.examples}
Our aim is to list some interesting examples of varieties arising as degeneracy loci that can be described by \Cref{thm.badlines}.

\begin{example}
Let us study in more detail the case $m=1$. Recall that we are only interested in applications with $n\geq3$ and $s\in\{3,4,5\}$.
\Cref{thm.badlines} above proves that the map $\psi$ does not contract lines inside $S$ precisely when one of the following conditions is satisfied:
\begin{itemize}
\item $n\geq 3$ for curves in $\mathbb{P}^3$ (we excluded the case of the twisted cubic obtained for $n=2$),
\item $n\geq 4$ for surfaces in $\mathbb{P}^4$,
\item $n\geq 6$ for threefolds in $\mathbb{P}^5$.
\end{itemize}
Moreover, again by \Cref{thm.badlines}, under the assumption $n\geq3$ the map $\psi$ does not admit $2$-dimensional fibres.
\end{example}

\begin{example}[White surfaces]\label{ex.whitesurfaces}
Fix $m\geq 0$ and choose $s=m+3$ and $n=s-m=3$. Now, the degeneracy locus $\mathsf{S}_m$ is a surface in $\PP^{m+3}$. Moreover, by \Cref{thm.badlines} the map $\psi\colon \mathsf{S}_m\to \PP^{2}$ is surjective and generically injective. The exceptional divisor (i.e.~the union of the $1$-dimensional fibres) arises over a $0$-dimensional locus so that $\mathsf{S}_m$ is the blow up of $\PP^2$ at $c$ points. Again by \Cref{thm.badlines}, $c$ can be easily computed as the degree of $N_{s-1}$ in $\PP(\mathsf{Mat}_{s,s+1}(\mathbb C))$, namely
\[ 
c = \frac{(s+1)!}{(s-1)!2!} = \binom{m+4}{2} \; . 
\]
We also observe the following:
\begin{itemize}
    \item For $m=0$ we obtain the determinantal cubic surface $\mathsf{S}_0\subset\PP^3$ realised as the blow up of $\PP^2$ in $6$ points.
    \item For $m=1$ we recover the classical construction of the Bordiga surface $\mathsf{S}_1\subset \PP^4$ realised as the blow up of $\PP^2$ in $10$ points, see e.g.~\cite{Ottaviani_1}.
    In this case $e_{\mathrm{top}}(Z)= 94$, with $h^{1,1}=12$, $h^{2,2}=68$ and the other relevant Hodge numbers being $0$. On the other hand, $\Hilb^2(\mathsf{S}_1)$ has topological Euler characteristic $104$, with $h^{2,2}=78$.
\end{itemize}
In the general case $\mathsf{S}_m\subset\PP^{m+3}$ is nothing but the $(m+3)$-th White surface named after F. Puryer White, see~\cite{white1923}.
\end{example}

\begin{example}[Generalised Bordiga scrolls over $\PP^2$]\label{ex.bordigascroll}
Fix $m\geq 1$ and choose $s=m+4$ and $n=s-m-1=3$. Notice that the condition $m\geq 1$ ensures that $s\in\{m+2\,,\dots,\, 2m+3\}$. In this case the degeneracy locus $\mathsf{B}_m$ is a threefold in $\PP^{m+4}$. Since the fibre of the map $\psi\colon \mathsf{B}_m\to \PP^{2}$ is cut by $n+m=s-1$ equations, the generic fibre of $\psi$ is $1$-dimensional. On the other hand, following the same argument of the proof of \Cref{thm.badlines} it is immediate to see that $2$-dimensional fibres of $f\circ\psi$ may only arise over $\mathrm{Im}(f)\cap N_{s-3}=\emptyset$, being $\codim(N_{s-3})=8$. Hence the map $\psi$ is surjective and realises $\mathsf{B}_m\subset\PP^{m+4}$ as a $\PP^1$-bundle over $\PP^2$, so that $\mathsf{B}_m$ is the projectivisation of a rank $2$ vector bundle over $\PP^2$.

In particular, for $m=1$ we recover the classical construction of the Bordiga scroll $\mathsf{B}_1\subset \PP^5$, i.e.~the (rational, non Fano) variety described by $\PP_{\PP^2}(E)$, with $E$ a rank 2 stable bundle with $c_1(E)=4$, $c_2(E)=10$, see e.g.~\cite{ottaviani19923}.
\end{example}

We were not able to find a precise reference for the threefolds described in \Cref{ex.bordigascroll}, so that we decided to call these threefolds \emph{generalised Bordiga scrolls}, in analogy with the classical Bordiga scroll, see e.g.~\cite{ottaviani19923}.

\begin{example}[White varieties]\label{ex.whitevarieties}
Choose $m\geq0$, and $3\leq n\leq m+3$. Fix $s=n+m\in\{m+3\,,\dots,\,2m+3\}$. Denote by $\mathsf{W}_{m,n}=D_{n-1}(\varphi)$ the usual degeneracy locus. Then by \Cref{thm.badlines} the map $\psi\colon \mathsf{W}_{m,n}\to \PP^{n-1}$ is surjective and generically injective. In particular, $\dim \mathsf{W}_{m,n} = n-1$. Moreover, $1$-dimensional fibres arise over an $(n-3)$-dimensional locus.

We also observe the following:
\begin{itemize}
    \item For any $m\geq0$, $\,\mathsf{W}_{m,3}$ is nothing but the $(m+3)$-th White surface denoted by $\mathsf{S}_m$ in \Cref{ex.whitesurfaces}.
    \item For any $m\geq1$, $\,\mathsf{W}_{m,4}\subset\PP^{m+4}$ is a threefold that contains a $\PP^1$-scroll over a curve $\mathsf{W}_{m,4}'\subset \PP^{3}$. We can also compute the degree and the genus of $\mathsf{W}_{m,4}'$ as
    \begin{align*}
        \deg(\mathsf{W}_{m,4}')&=\deg(N_{m+3})=\binom{m+5}{2} \\
        g(\mathsf{W}_{m,4}')&= (m+4)\binom{m+4}{3}-(m+5)\binom{m+3}{3},
    \end{align*}
as proved in \Cref{prop.genus}.
\end{itemize}
\end{example}

We were not able to find a precise reference for the construction spelled out in \Cref{ex.whitevarieties}, so we decided to call these $(n-1)$-folds \emph{White varieties}, in analogy with the usual White surfaces described in \Cref{ex.whitesurfaces}.

Apart from the limit case of White varieties ($s=n+m$) we provide examples for which $s\leq n+m$ but $Z$ need not to be isomorphic to $\Hilb^2(S)$.
More precisely, it may be interesting to investigate the limit case when $n=2s-2m-3$. Notice that, given $m\geq 0$, assuming $s\in\{m+2,\dots,2m+3\}$ the system
\[ \begin{cases}
s\leq n+m\\
n=2s-2m-3\geq 3
\end{cases} \]
implies $s\geq m+3$ and $n\leq 2m+3$.

\begin{example}[$n=2s-2m-3$]\label{example.3foldm1}
Fix $m\geq0$, $s\in\{m+3,\ldots,2m+3\}$, and choose $n=2s-2m-3$. By \Cref{thm.badlines} the map $\psi$ maps the degeneracy locus $\mathsf{M}_{m,s}\subset \PP^s$ onto a certain variety $\mathsf{M}_{m,s}'\subset\PP^{n-1}$ of dimension $s-m-1$, having a finite number $c$ of special points over which the fibres are $1$-dimensional. Notice that it is easy to compute $c$, since it equals the degree of $N_{s-1}$ inside $\mathbb{P}(\mathsf{Mat}_{n+m\,,\,s+1}(\mathbb C))$, which is given by the formula \cite{Ottaviani_1}
\begin{align*}
c&=\deg(N_{s-1})= \prod_{i=0}^1\frac{(n+i+m)!\, i!}{(s+i-1)!\, (n+m-s+i+1)!} \\
&=\frac{1}{s}\,\binom{2s-m-2}{s-1}\,\binom{2s-m-3}{s-1}.
\end{align*}
Remarkably, the same proof of \Cref{thm.verybadlines} excludes lines of type $\boldit{d}$ as soon as $m\geq 3$.
Notice that the choice $s=m+3$ implies $n=3$, so that in particular we recover the White surfaces described in \Cref{ex.whitesurfaces}, i.e.~$\mathsf{M}_{m,m+3}=\mathsf{S}_{m}$.
\begin{description}
    \item[$m=0$] In this case we only have the White surface $\mathsf{M}_{0,3}=\mathsf{S}_0\subset \PP^3$.
    \item[$m=1$] In this case, apart from the Bordiga surface $\mathsf{M}_{1,4}=\mathsf{S}_1$ already discussed in \Cref{ex.whitesurfaces}, we may only choose $s=n=5$. Then $\mathsf{M}_{1,5}$ is a threefold in $\PP^5$. By the above formula there are $c=105$ fibres of dimension $1$ and the image of $\mathsf{M}_{1,5}$ inside $\mathbb{P}^4$ is a determinantal threefold $\mathsf{M}_{1,5}'=f^{-1}(\mathrm{Im}(f)\cap N_5)$ of degree $6$, whose singular locus consists exactly of these 105 points. In fact, $\mathsf{M}_{1,5}$ is a small resolution of $\mathsf{M}_{1,5}'$. Notice how in this case $e_{\mathrm{top}}(Z)=46158$ and $e_{\mathrm{top}}(\Hilb^2(S_{1,5}))=46053$ by \Cref{sec:hodge_3fold}. Their difference is exactly $105$ so that in particular $Z\not\cong\Hilb^2(\mathsf{M}_{1,5})$.
    \item[$m=2$] In this case, apart from the White surface $\mathsf{M}_{2,5}=\mathsf{S}_2$, we may choose  $s=6$ and $n=5$ or $s=n=7$. Now, by \Cref{thm.badlines}, $\mathsf{M}_{2,6}\subset\PP^6$ is a threefold, and the map $\psi$ contracts $c=\frac{1}{6}\binom{8}{5}\binom{7}{5}=196$ lines. On the other hand $\mathsf{M}_{2,7}\subset\PP^6$ is a fourfold, and the map $\psi$ contracts $c=\frac{1}{7}\binom{10}{6}\binom{9}{6}=2520$ lines.
\end{description}
\end{example}

\Cref{example.3foldm1} leads us to formulate the following conjecture.

\begin{conjecture}\label{conj.optimalbound}
Fix $m\geq1$, $s\in\{m+3,\ldots,2m+3\}$, and choose $n=2s-2m-3$. Then
\[ 
e_{\mathrm{top}}(\Hilb^2(\mathsf{M}_{m,s})) -e_{\mathrm{top}}\left(Z_{n,s,m}\right) = (-1)^{\dim(\mathsf{M}_{m,s})}\, \frac{1}{s}\,\binom{2s-m-2}{s-1}\,\binom{2s-m-3}{s-1} \; . 
\]
\end{conjecture}

Notice that in \Cref{ex.whitesurfaces} and \Cref{example.3foldm1} we have shown that the above conjecture holds true for $m=1$. We also did the computation for White surfaces taking higher values of $m$ confirming the prediction of \Cref{conj.optimalbound}.

On the other hand we excluded the case $m=0$, for which the conjecture is easily seen to fail. However, this can be justified by noticing that $\mathsf{M}_{0,3}=\mathsf{S}_0$ contains $15$ lines of type $\boldit{d}$ (arising as birational transforms of lines in $\PP^2$ passing through $2$ out of the $6$ points of $\PP^2$), and indeed we compute the difference to be
\[ 
e_{\mathrm{top}}(\Hilb^2(\mathsf{M}_{0,3})) -e_{\mathrm{top}}(Z_{3,3,0}) = 6+15 \; . 
\]
As already remarked in \Cref{example.3foldm1} it is immediate to see that for $m\geq3$ the varieties $\mathsf{M}_{m,s}$ do not admit lines of type $\boldit{d}$, and actually we do not expect this to happen even in the cases $m=1$ and $m=2$.

We are particularly interested in \Cref{conj.optimalbound} since it would imply for instance that the bound provided by \Cref{thm.Hilbiso} is optimal.

\appendix

\section{Euler characteristic of Hilbert squares} \label{sec:Euler}
The goal of this appendix is to give a detailed proof of \Cref{prop:euler_char}. We shall exploit a nontrivial Chern class calculation on (smooth) degeneracy loci following Pragacz \cite{Pragacz}.

Fix $m = 1$ throughout this section. Let $s \in \{3,4\}$, and consider, as ever, a general map $\varphi\colon \mathcal F \to \mathcal E$ between vector bundles $\mathcal F=\of_{\PP^s}^{\oplus n+1}$ and $\mathcal E=\of_{\PP^s}(1)^{\oplus n}$. The $k$-th degeneracy locus of $\varphi$ is the closed subscheme $D_k(\varphi) \subset \PP^s$ defined by the condition $\rank (\varphi) \leq k$, which is (locally) equivalent to the vanishing of the $(k+1)$-minors of $\varphi$. We are interested in the case $k=n-1$, which leads to $D_{n-2}(\varphi)$ of expected codimension $6$, and $D_{n-1}(\varphi)$ of expected codimension $2$. Since $\varphi$ is general, we have $D_{n-2}(\varphi)=\emptyset$, so that $D_{n-1}(\varphi) \subset \PP^s$ is a smooth subvariety of codimension $2$. In the case $s=4$, we shall denote it by $S_n \subset \mathbb P^4$, whereas in the case $s=3$ we shall denote it by $C_n \subset \mathbb P^3$.

We start assuming $s=4$, the case $s=3$ being essentially a truncation of the case $s=4$. Let $H \in A^1(\PP^4)$ denote the first Chern class of $\mathcal O_{\PP^4}(1)$. The ordinary Segre class of $\mathcal E$ is the class
\[
\widetilde s(\mathcal E) 
= \sum_{0\leq i\leq 4}\widetilde s_i(\mathcal E) = (1+H)^{-n},
\]
with $\widetilde s_i(\mathcal E) \in A^i(\PP^4) = \mathbb Z[H^i]$ sitting in codimension $i$.
Inverting the Chern class
\[
c(\mathcal E) = 1+nH+\binom{n}{2}H^2+\binom{n}{3}H^3+\binom{n}{4}H^4
\]
we find
\begin{align*}
    \widetilde s_1(\mathcal E) &= -c_1(\mathcal E)=-nH \\
    \widetilde s_2(\mathcal E) &= s_1(\mathcal E)^2-c_2(\mathcal E) =  \left[n^2-\binom{n}{2}\right] H^2\\
    \widetilde s_3(\mathcal E) &= -s_1(\mathcal E)c_2(\mathcal E)-s_2(\mathcal E)c_1(\mathcal E)-c_3(\mathcal E) =  \left[-n^3-\binom{n}{3}+2n\binom{n}{2} \right]H^3\\
    \widetilde s_4(\mathcal E) &= -s_1(\mathcal E)c_3(\mathcal E)-s_2(\mathcal E)c_2(\mathcal E)-s_3(\mathcal E)c_1(\mathcal E)-c_4(\mathcal E) \\ 
    &= \left[n^4+2n\binom{n}{3}-3n^2\binom{n}{2}+\binom{n}{2}^2-\binom{n}{4} \right] H^4.
\end{align*}
We set $s_i = (-1)^i \widetilde{s_i}(\mathcal E)$ for $0\leq i\leq 4$.
Then, unraveling \cite[Example 5.8~(ii)]{Pragacz}, we have, for the smooth surface $S_n \subset \mathbb P^4$, an identity 
\begin{equation}\label{eqn:pragacz}
    e_{\mathrm{top}}(S_n) =
    s_2c_2(\PP^4)-\left[s_{(2,1)}+2s_3\right]c_1(\PP^4)+s_{(2,1,1)}+3s_{(3,1)}+3s_4,
\end{equation}
given the Schur polynomials
\begin{align*}
    s_{(2,1)} &= 
    \begin{vmatrix}
    s_2 & s_3 \\
    s_0 & s_1
    \end{vmatrix}
    = 
    \begin{vmatrix}
    s_2 & s_3 \\
    1 & s_1
    \end{vmatrix} = s_2s_1-s_3
    \\
    s_{(3,1)} &=
    \begin{vmatrix}
    s_3 & s_4 \\
    s_0 & s_1
    \end{vmatrix}
    =
    \begin{vmatrix}
    s_3 & s_4 \\
    1 & s_1
    \end{vmatrix} = s_3s_1-s_4\\
    s_{(2,1,1)} &= 
    \begin{vmatrix}
    s_2 & s_3 & s_4 \\
    s_0 & s_1 & s_2 \\
    0 & s_0 & s_1
    \end{vmatrix}
    = 
    \begin{vmatrix}
    s_2 & s_3 & s_4 \\
    1 & s_1 & s_2 \\
    0 & 1 & s_1
    \end{vmatrix}=s_2(s_1^2-s_2)-(s_1s_3-s_4).
\end{align*}
Expanding, we obtain
\begin{align*}
    s_2c_2(\PP^4) &= 10n^2-10\binom{n}{2} \\
    \left[s_{(2,1)} + 2s_3\right] c_1(\PP^4) &= 5(s_2s_1+s_3)H = 10n^3 - 15n\binom{n}{2}+5\binom{n}{3} \\
    s_{(2,1,1)} 
    &= n\binom{n}{3}-\binom{n}{4} \\
    3s_{(3,1)} &= \binom{n}{2}\left[3n^2-3\binom{n}{2}\right]-3n\binom{n}{3}+3\binom{n}{4} \\
    3s_4 &= 3n^4+6n\binom{n}{3}-9n^2\binom{n}{2}+3\binom{n}{2}^2-3\binom{n}{4}.
\end{align*}
Formula \eqref{eqn:pragacz} then yields
\[
 e_{\mathrm{top}}(S_n)
 =n^2(10-10n+3n^2)+\binom{n}{2}(-10+15n-6n^2)+\binom{n}{3}(4n-5)-\binom{n}{4}.
\]
In the case of a smooth determinantal \emph{curve} $C_n \subset \PP^3$, i.e.~when we set $s=3$, we only need to use
\[
s_0=1,\quad s_1=nH,\quad s_2 = \left[n^2-\binom{n}{2}\right] H^2,\quad s_3= \left[n^3+\binom{n}{3}-2n\binom{n}{2}\right]H^3. 
\]
In this case, \cite[Example 5.8~(i)]{Pragacz} gives
\begin{align*}
e_{\mathrm{top}}(C_n) 
&= s_2c_1(\PP^3)-s_{(2,1)}-2s_3 
= 4Hs_2 - (s_2s_1-s_3) -2s_3 
= 4Hs_2 - s_2s_1 - s_3 \\
&= 4n^2-4\binom{n}{2}-n^3+n\binom{n}{2}-n^3-\binom{n}{3}+2n\binom{n}{2} \\
&= 4n^2-2n^3+(3n-4)\binom{n}{2}-\binom{n}{3}.
\end{align*}
The formulas for $e_{\mathrm{top}}(S_n)$ and $e_{\mathrm{top}}(C_n)$ prove \Cref{prop:euler_char}.

\section{Hodge--Deligne polynomial  of Hilbert squares}\label{sec:hodge-poly}
We again set $m=1$ throughout this section. We shall consider once more smooth (sub-determinantal) degeneracy loci $S = D_{n-1}(\varphi)\subset \mathbb P^s$ (of dimension 2 or 3), and we shall compute the Hodge--Deligne polynomial
\[
E(\Hilb^2(S);u,v) = \sum_{p,q\geq 0}h^{p,q}(\Hilb^2(S))(-u)^p(-v)^q \,\in\,\mathbb Z[u,v]
\]
via standard motivic techniques, exploiting the power structure on the Grothendieck ring of varieties $K_0(\Var_{\mathbb C})$ \cite{GLMHilb}, as well as our knowledge of the Hodge numbers of $S$ (cf.~\Cref{sec:ex}).

\subsection{Surface case: \texorpdfstring{$(s,n,m)=(4,4,1)$}{}}
Let us consider the smooth determinantal surface $S_4 = D_{3}(\varphi) \subset \PP^4$. By G\"{o}ttsche's formula \cite{Gott1} for the motive of the Hilbert scheme of points on a surface, combined with the main result of \cite{GLMHilb}, there is an identity
\[
\sum_{n\geq 0}\left[\Hilb^n (S_4)\right] q^n = \prod_{n>0}\left(1-\mathbb L^{n-1}q^n \right)^{-[S_4]}
\]
in $K_0(\Var_{\mathbb C})\llbracket q \rrbracket$, where exponentiation is to be thought of in the language of power structures. The Hodge--Deligne polynomial of a smooth projective $\mathbb C$-variety $Y$ is the polynomial
\[
E(Y;u,v) = \sum_{p,q \geq 0}h^{p,q}(Y) (-u)^p(-v)^q \,\in\,\mathbb Z[u,v].
\]
We have, on $\mathbb Z[u,v]$, the power structure defined by the identity
\[
\left(1-q\right)^{-f(u,v)} = \prod_{i,j}\left(1-u^iv^jq\right)^{-p_{ij}}
\]
if $f(u,v) = \sum_{i,j}p_{ij}u^iv^j$. Looking at the Hodge diamond depiced in \Cref{sec:45}, we deduce
\[
E(S_4;u,v) = 1+4u^2+45uv+4v^2+u^2v^2,
\]
and since $E(-)$ defines a morphism $K_0(\Var_{\mathbb C}) \to \mathbb Z[u,v]$ of \emph{rings with power structure} sending $\mathbb L\mapsto uv$, we have an identity
\begin{align*}
    \sum_{n\geq 0} E(\Hilb^n(S_4);u,v) q^n 
    &= \prod_{n>0} \left(1-u^{n-1}v^{n-1}q^n \right)^{-E(S_4;u,v)} \\
    &= \prod_{n>0} \left(1-q\right)^{-E(S_4;u,v)}\big|_{q \mapsto u^{n-1}v^{n-1}q^n} \\
    &= \prod_{n>0} \left(1-u^{n-1}v^{n-1}q^n\right)^{-1} \left(1-u^{n+1}v^{n-1}q^n \right)^{-4}\cdot \\ & \qquad \qquad \cdot\left(1-u^{n}v^{n}q^n \right)^{-45}\left(1-u^{n-1}v^{n+1}q^n \right)^{-4}\left(1-u^{n+1}v^{n+1}q^n \right)^{-1}
\end{align*}
where the substitution $q \mapsto u^{n-1}v^{n-1}q^n$ is possible thanks to the properties of a power structure.

Expanding and isolating the coefficient of $q^2$ gives 
\begin{multline*}
    E(\Hilb^2(S_4);u,v) 
    =1+46uv+4(u^2+v^2)+1097u^2v^2+184(uv^3+u^3v)+10(u^4+v^4)\\
    +46u^3v^3+4(u^4v^2+u^2v^4)
    +u^4v^4,
\end{multline*}
in full agreement with the Hodge diamond depicted in \Cref{sec:45}.

\subsection{Threefold case: \texorpdfstring{$(s,n,m)=(5,5,1)$}{}}\label{sec:hodge_3fold}
In the case $(s,n,m)=(5,5,1)$, we obtain a smooth threefold $S_{5,5,1} \subset \PP^5$ outside the `good range' of \Cref{thm:mainthm_intro}, cf.~\Cref{example.3foldm1}. There is an identity \cite{GLMHilb,ricolfi2019motive}
\[
\mathsf Z_{S_{5,5,1}}(q) = \sum_{n\geq 0}\,\left[\Hilb^n(S_{5,5,1})\right]q^n=\left(\sum_{n\geq 0}\,\bigl[\Hilb^n(\mathbb A^3)_0\bigr]q^n\right)^{[S_{5,5,1}]}
\]
in $K_0(\Var_{\mathbb C})\llbracket q \rrbracket$, where $\Hilb^n(\mathbb A^3)_0$ denotes the punctual Hilbert scheme, namely the subscheme of $\Hilb^n(\mathbb A^3)$ parametrising subschemes entirely supported at the origin $0 \in \mathbb A^3$. Let us define classes $\Omega_n \in K_0(\Var_{\mathbb C})$ via the relation
\[
\sum_{n\geq 0}\bigl[\Hilb^n(\mathbb A^3)_0\bigr]q^n = \Exp\left( \sum_{n>0}\Omega_nq^n\right)=\prod_{n>0}\,\left(1-q^n\right)^{-\Omega_n}.
\]
Since $\Hilb^1(\mathbb A^3)_0=\Spec\mathbb C$ and $\Hilb^2(\mathbb A^3)_0 = \PP^2$, one can easily compute $\Omega_1=1$ and $\Omega_2=\mathbb L+\mathbb L^2$. Therefore
\[
\mathsf Z_{S_{5,5,1}}(q) = \prod_{n>0}\, \left(1-q^n\right)^{-\Omega_n[S_{5,5,1}]},
\]
which implies
\begin{equation}\label{hodge_generating_series}
\sum_{n\geq 0}E(\Hilb^n(S_{5,5,1});u,v)q^n = \prod_{n>0}\, \left(1-q^n\right)^{-E(\Omega_n;u,v)E(S_{5,5,1};u,v)}.
\end{equation}
One can compute the Hodge--Deligne polynomial of $S_{5,5,1}$ to be
\[
E(S_{5,5,1};u,v) = 1+2uv+2u^2v^2+u^3v^3-(5u^3+151u^2v+151uv^2+5v^3),
\]
so that extracting the coefficient of $q^2$ from \eqref{hodge_generating_series}, one obtains
\[
E(\Hilb^2(S_{5,5,1});u,v)=
\left[\frac{(1-u^3q)^5(1-v^3q)^5(1-uv^2q)^{151}(1-u^2vq)^{151}}{(1-q)(1-uvq)^2(1-u^2v^2q)^2(1-u^3v^3q)}\right]_{q^2} + (uv+u^2v^2)E(S_{5,5,1};u,v).
\]
In particular, the topological Euler characteristic is 
\[
e_{\mathrm{top}}(\Hilb^2(S_{5,5,1})) = E(\Hilb^2(S_{5,5,1});1,1) = 46053 = e_{\mathrm{top}}(Z_{5,5,1})-105.
\]

\subsection{Threefold case: \texorpdfstring{$(s,n,m)=(5,6,1)$}{}}
In the case $(s,n,m)=(5,6,1)$, we get a smooth threefold $S_{5,6,1} \subset \PP^5$. Using the Hodge diamond depicted in \Cref{sec:651}, one has
\[
E(S_{5,6,1};u,v)=1+2uv+2u^2v^2+u^3v^3-(29u^3+520u^2v+520uv^2+29v^3).
\]
Formula \eqref{hodge_generating_series} applied to this case yields
\[
E(\Hilb^2(S_{5,6,1});u,v)=
\left[\frac{(1-u^3q)^{29}(1-v^3q)^{29}(1-uv^2q)^{520}(1-u^2vq)^{520}}{(1-q)(1-uvq)^2(1-u^2v^2q)^2(1-u^3v^3q)}\right]_{q^2} + (uv+u^2v^2)E(S_{5,6,1};u,v).
\]
In particular,
\[
e_{\mathrm{top}}(\Hilb^2(S_{5,6,1})) = E(\Hilb^2(S_{5,6,1});1,1) = 593502,
\]
in complete agreement with what one gets out of the Hodge diamond for $Z$ depicted in \Cref{sec:651}.

\bibliographystyle{amsplain-nodash}
\bibliography{bib}

\end{document}